\documentclass[11pt, reqno, english]{amsart}
\usepackage{accents}
\usepackage{style}
\usepackage[normalem]{ulem}
\usepackage{wasysym}
\usepackage{subcaption}
\usepackage{amsmath}

\usepackage{charter}
\usepackage{euler}
\usepackage[T1]{fontenc}
\usepackage[bb=stixtwo]{mathalpha}
\usepackage{tikz}
\usepackage[color,matrix,arrow]{xy}

\usepackage{thmtools,thm-restate}

\def\cR{\mathcal R}

\renewcommand{\dis}{\mathrm{dis}}
\renewcommand{\codis}{\mathrm{codis}}
\DeclareMathOperator{\sep}{sep}

\begin{document}

\title{Lower bounding the Gromov--Hausdorff distance in metric graphs}

\author{Henry Adams}
\email{henry.adams@ufl.edu}
\author{Sushovan Majhi}
\email{s.majhi@gwu.edu}
\author{Fedor Manin}
\email{manin@math.toronto.edu}
\author{\v{Z}iga Virk}
\email{ziga.virk@fri.uni-lj.si}
\author{Nicol\`o Zava}
\email{nicolo.zava@ist.ac.at}

\begin{abstract}
Let $G$ be a finite, connected metric graph and let $X\subseteq G$ be a subset.
If $X$ is sufficiently dense in $G$, we show that the Gromov--Hausdorff distance matches the Hausdorff distance, namely $d_\gh(G,X)=d_\h(G,X)$.
When the metric graph is the circle $G=S^1$ with circumference $2\pi$, a recent study established the equality $d_\gh(S^1,X)=d_\h(S^1,X)$ whenever $d_\gh(S^1,X)<\frac{\pi}{6}$.
Our results relax this hypothesis to $d_\gh(S^1,X)<\frac{\pi}{3}$, and furthermore, we show that the constant $\frac{\pi}{3}$ is the best possible.
We lower bound the Gromov--Hausdorff distance $d_\gh(G,X)$ by the Hausdorff distance $d_\h(G,X)$ via a simple topological obstruction: the existence of a possibly discontinuous function $f\colon G \to X$ with too small distortion contradicts the connectedness of $G$.
\end{abstract}



\maketitle

\setcounter{tocdepth}{1}
\tableofcontents

\section{Introduction}

The Gromov--Hausdorff distance between two abstract metric spaces $(X,d_X)$ and $(Y,d_Y)$, denoted $d_\gh(X,Y)$, provides a dissimilarity measure quantifying how far the two metric spaces are from being isometric~\cite{edwards1975structure,gromov1981structures,tuzhilin2016invented}.
In the past two decades, this distance has found applications in topological data analysis (TDA) as a theoretical framework for shape and dataset comparison~\cite{memoli2007use}, which motivated the study of its quantitative aspects.
However, precise computations of the Gromov--Hausdorff distance between even simple spaces are mostly unknown, with certain exceptions such as the Gromov--Hausdorff distance between a line segment and a circle~\cite{ji2021gromov} and between spheres of certain dimensions~\cite{lim2023gromov,GH-BU-VR,harrison2023quantitative}.
Even approximating the Gromov--Hausdorff distance by a factor of $3$ in the case of trees with unit-edge length is known to be NP-hard~\cite{GH_trees_NP,schmiedl2017computational}.
A better approximation factor can only be achieved in very special cases; for example,~\cite{majhi2023GH} provides a $\frac{5}{4}$-approximation scheme if $X$ and $Y$ are finite subsets of the real line.
Developing tools and pipelines to estimate the Gromov--Hausdorff distance is a central research direction in the computational topology community.

To estimate the Gromov--Hausdorff distance between two metric spaces $X$ and $Y$, one may:\\
\textbf{(A)}
Find sufficiently nice samples $X'\subseteq X$ and $Y'\subseteq Y$ approximating the geometry of $X$ and $Y$,\\ 
\textbf{(B)}
Efficiently bound the Gromov--Hausdorff distance between the subsets $X'$ and $Y'$.\\
Indeed, when $X$ and $Y$ are infinite continuous objects---e.g., Riemannian manifolds, metric graphs---and one wants to use computational machinery, it is often essential to resort to finite subsets approximating the geometry of the original spaces.

We reformulate the first task \textbf{(A)} as follows: how dense does the sample $X'$ have to be in $X$ to capture the geometry of the ambient space $X$? 
To systematize this question, we use the Hausdorff distance as a measure of density and ask if we can provide a threshold $\varepsilon$ and constant $0<C\leq1$ such that $d_\gh(X,X')\geq C\cdot d_\h(X,X')$ whenever $d_\h(X,X')<\varepsilon$.
In particular, $d_\gh(X,X')=d_\h(X,X')$ if $C=1$.
In~\cite{HvsGH}, the authors show that when $X$ is a manifold and $d_\gh(X,X')<\varepsilon$ (which implies $d_\h(X,X')<\varepsilon$) for some constant $\varepsilon$, then $d_\gh(X',X)\geq C\cdot d_\h(X,X')$ for some $\frac{1}{2}\le C\le 1$.
Our study continues that investigation and proves improved results (such as with $C=1$) when $X$ is a metric graph, hence relaxing the manifold assumption.

A fruitful approach to provide bounds on the Gromov--Hausdorff distance in task \textbf{(B)} is to use stable metric invariants.
One associates each metric space $X$ with a value $\psi(X)$ in another computationally simpler metric space so that $\psi(X)$ and $\psi(Y)$ are close whenever $X$ and $Y$ are close in the Gromov--Hausdorff distance.
An immediate example is the diameter; we have $d_\gh(X,Y)\ge\frac{1}{2}|\diam(X)-\diam(Y)|$.
However, the diameter bound is often not informative when the metric spaces are of similar diameter but metrically very different.
For further stable metric invariants, we refer the reader to~\cite{memoli2012some}.
Other central examples in computational topology are Vietoris--Rips persistence diagrams~\cite{ChazalDeSilvaOudot2014} and hierarchical clustering~\cite{MemoliHierarchical2010}.
In~\cite{zava_embed} dimension theory is used to provide unavoidable limits to the precision of metric invariants with values in Hilbert spaces.

One of the precursors of algebraic topology was Brouwer's proof of \emph{invariance of dimension}, namely that $\R^m$ is not homeomorphic to $\R^n$ for $n>m$.
Proving $\R^1$ is not homeomorphic to $\R^n$ for $n>1$ relies only on connectedness: removing a single point from $\R^1$ leaves a disconnected space, which is not the case for $\R^n$ with $n>1$.
However, to prove invariance of dimension for $n>m\ge 2$, Brouwer relied on concepts related to the \emph{degree} of a map between spheres~\cite{brouwer1911abbildung,brouwer1911beweis}, which is often now formalized using higher-dimensional homotopy groups or homology groups.
The topological obstructions to small Gromov--Hausdorff distances in~\cite{HvsGH} relied on homological tools, namely the fundamental class of a manifold.
One of the key contributions of our paper is to go ``back in time'' and produce strong lower bounds on Gromov--Hausdorff distances using simpler topological tools: we show that small Gromov--Hausdorff distances would contradict the connectedness of a space.

\subsection*{Overview of our results}
Let $G$ be a connected, finite metric graph, and let $X \subseteq G$ be a subset.
We show how to lower bound the Gromov--Hausdorff distance $d_\gh(G,X)$ in terms of a constant multiplier times the Hausdorff distance $d_\h(G,X)$, namely $d_\gh(G,X) \ge C\cdot d_\h(G,X)$.
We often obtain the \emph{optimal} value $C=1$.

A unifying theme in our proofs is that we show that a possibly discontinuous function $f\colon G\to X$ induces a continuous and non-surjective map $\overline{f}\colon G\to G$.
We then use non-surjectivity (and related properties) to lower bound the distortion of $\overline{f}$ and $f$, hence lower bounding $d_\gh(G,X)$.

The simplest metric graph is a tree $T$.
Let $\partial T$ denote the leaves of $T$, and let $\vec{d}_\h(\partial T, X)$ denote the directed Hausdorff distance from the leaves to the sample $X$ (see Section~\ref{sec:background}).
We show the following.

\begin{theorem_maps_and_trees}
Let $T$ be a finite metric tree and let $X\subseteq T$.
If $d_\h(T,X)>\vec d_\h(\partial T,X)$, then $d_\gh(T,X)=d_\h(T,X)$.
\end{theorem_maps_and_trees}

The interpretation is as follows: if the biggest gaps in the sampling $X\subseteq T$ are not near the leaves of $T$ (which by Proposition~\ref{prop:assumption_on_leaves_generic} is typical behavior), then $d_\gh(T, X)=d_\h(T, X)$.
The following theorem shows that some control on the set of leaf vertices of $T$ is needed.

\begin{theorem_counterExample}
For any $\varepsilon > 0$ there exists a finite metric tree $T$ and a closed subset $X\subseteq T$ such that $d_\gh(T, X) \leq \varepsilon \cdot d_\h(T, X)$.
\end{theorem_counterExample}

When $G=S^1$ is the circle of circumference $2\pi$,~\cite{HvsGH} shows that if $d_\gh(S^1, X)<\frac{\pi}{6}$, then $d_\gh(S^1, X) = d_\h(S^1, X)$.
The following theorem improves the constant $\frac{\pi}{6}$ to $\frac{\pi}{3}$.

\begin{theorem_circleSufficient}
For any subset $X \subseteq S^1$, we have $d_\gh(S^1, X) \ge \min\left\{ d_\h(S^1, X),\ \tfrac{\pi}{3} \right\}$.
\end{theorem_circleSufficient}

\noindent In Theorem~\ref{thm:circleOptimal} we show that this constant $\frac{\pi}{3}$ is the best possible, namely that for any $\varepsilon\in\left(0,\frac{\pi}{6}\right)$, there exists a nonempty compact subset $X\subseteq S^1$ with $d_\h(S^1, X)=\frac{\pi}{3}+\varepsilon$ and $d_\gh(S^1, X)=\frac{\pi}{3}<d_\h(S^1, X)$.

We generalize the above results for the circle and trees to general metric graphs.
Let $G$ be a finite metric graph with set of leaf vertices $\partial G$.
We prove that if $X$ is dense enough relative to $e(G_0)$, which is the smallest edge length between vertices in the core $G_0$ of $G$ (see Section~\ref{sec:background}), then the Gromov--Hausdorff distance matches the Hausdorff distance.

\begin{theorem_graphs}
Let $G$ be a finite connected metric graph and let $X\subseteq G$.
If $\partial G\neq \emptyset$, suppose $d_\h(G,X)>\vec{d}_\h(\partial G,X)$.
If $d_\gh(G,X)<\tfrac{e(G_0)}{24}$, then $d_\gh(G,X)= d_\h(G,X)$.
\end{theorem_graphs}

Theorems~\ref{thm:circleSufficient} and~\ref{thm:graphs} are obtained using low-dimensional versions of the Borsuk--Ulam theorem for maps from the circle into $\R$ or into a tree, which can be proven using path-connectedness.

All of these results have corollaries allowing us to lower bound the Gromov--Hausdorff distance between two subsets $X, Y \subseteq G$, using the triangle inequality.
That is, for two sufficiently dense subsets $X,Y\subseteq G$, we provide new lower bounds on $d_\gh(X,Y)$ in terms of the Hausdorff distance $d_\h(X,Y)$, which is $\mathcal{O}(n^2)$-time computable if the subsets have at most $n$ points.
We demonstrate such a corollary explicitly in the case of metric trees in Section~\ref{sec:trees}.

In Section~\ref{sec:no-boundary-assumptions} we provide variants of Theorem~\ref{thm:graphs}, resulting in weaker bounds on the Gromov--Hausdorff distance under weaker assumptions depending only on the structure of the metric graph: we no longer assume control over how close each leaf vertex of the graph is to a point of the sample.
The main results are Theorem~\ref{thm:trees_H_vs_GH_maximal_geodesics} for metric trees and Theorem~\ref{thm:graphs_H_vs_GH_maximal_geodesics} for general metric graphs, obtained by considering the lengths of maximal geodesics between leaf vertices.

\subsection*{Organization of the paper}
We begin in Section~\ref{sec:background} with background and notation.
In Section~\ref{sec:ratio}, we provide an example showing that some control over the boundary of the metric graph is necessary.
We begin with the case of metric trees in Section~\ref{sec:trees}, proceed next to the case of the circle in Section~\ref{sec:circle}, and consider general metric graphs in Section~\ref{sec:graphs}.
Finally, in Section~\ref{sec:no-boundary-assumptions} we lower bound the Gromov--Hausdorff distance between a metric graph and a subset thereof when we no longer assume that the sample is sufficiently close to the set of leaves, with variants of these results in Appendix~\ref{app:fractional}.

\section{Background and notation}
\label{sec:background}

We refer the reader to~\cite{BuragoBuragoIvanov,bridson2011metric} for more information on metric spaces, metric graphs, the Hausdorff distance, and the Gromov--Hausdorff distance.

\subsection{Metric spaces}
Let $(Z,d)$ be a metric space.
For any $z\in Z$, we let $B(z;r)=\{z'\in Z~|~d(z,z') < r\}$ denote the open metric ball of radius $r$ about $z$.
For any $r\geq 0$ and $A\subseteq Z$, we let
\[ A^r=\cup_{a\in A}B(a;r)=\left\{z\in Z\mid \inf_{a\in A}d(a, z)<r\right\} \]
denote the {\em $r$-thickening} of $A$ in $Z$, namely, the union of open metric balls of radius $r$ centered at the points of $A$.

The \emph{diameter} of a subset $A\subseteq Z$ is the supremum of all distances between pairs of points of $A$.
More formally, the diameter is defined as
\[\diam(A)=\sup\{d(a,a') \mid a,a'\in A\}.\]
The diameter of a compact set $A$ is always finite.

\subsection{Metric graphs}
To define metric graphs, we follow Definitions~3.1.12 (quotient metric) and~3.2.9 (metric graph) of~\cite{BuragoBuragoIvanov}.
A \emph{metric segment} is a metric space isometric to a real line segment $[0, l]$ for $l\geq0$.
Our metric graphs will always be connected and have a finite number of vertices and edges.

\begin{definition}[Metric Graph]\label{def:metric-graph}
A \emph{finite metric graph} $(G,d)$ is the metric space obtained by gluing a finite collection $E$ of metric segments along their boundaries to a finite collection $V$ of vertices, and equipping the result $\left(\bigsqcup_{e\in E}e\right)/\sim$ with the quotient metric.
We assume this equivalence relation results in a connected space.
\end{definition}

The metric segments $e\in E$ are called the \emph{edges} and the equivalence classes of their endpoints in $V$ are called the \emph{vertices} of $G$.
Since $G$ is connected, we have $d(y,y')<\infty$ for all $y,y'\in G$.
Since $E$ is finite, the quotient semi-metric~\cite[Definition~3.1.12]{BuragoBuragoIvanov} is in fact a metric.\footnote{
One needs to be more careful if $E$ is infinite.
Indeed, suppose one has two vertices $v$ and $v'$, and one edge $e_n$ of length $\frac{1}{n}$ with endpoints $v$ and $v'$ for each $n\in \mathbb{N}$.
Then the quotient semi-metric gives $d(v,v')=\inf_{n\in \mathbb{N}}\frac{1}{n}=0$, even though $v\neq v'$, and hence this semi-metric is not a metric.
These subtleties disappear since we assume $E$ is finite.}

For an edge $e\in E$, its length $|e|$ is the length of the corresponding metric segment.
Note that the length can be different from the distance between the endpoints of $e$ according to the metric $d$.
Extending the definition of length for a continuous path $\gamma:[a,b]\to G$, we define the length $|\gamma|$ as the sum of the (full or partial) lengths of the edges that $\gamma$ traverses.

For a vertex $v$ of a graph, its \emph{degree}, denoted $deg(v)$, is the number of edges incident to $v$.
A self-loop at $v$ contributes $2$ to the degree of $v$.
As an example, the circle $S^1$ can be given a metric graph structure with a single vertex $v$ of degree $2$ at the north pole, along with a self-loop.
We denote the set of leaf vertices, i.e.\ the vertices of degree one, by $\partial G$.

A metric graph $G$ is therefore endowed with two layers of information: a metric structure turning it into a length space~\cite[Definition~2.1.6]{BuragoBuragoIvanov} and a
combinatorial structure $(V, E)$ as an abstract graph.
Using the above definition, we are also allowing $G$ to have single-edge cycles and multiple edges between a pair of vertices.
In this paper, we only consider path connected metric graphs $(G,d)$ with finitely many edges.
Therefore, the metric graphs we consider are compact, and since a compact length space is geodesic, they are also geodesic (meaning for every $x,y\in G$ there is an isometric embedding $\gamma\colon[0,d(x,y)]\to G$ such that $\gamma(0)=x$ and $\gamma(d(x,y))=y$).
This justifies why for any $y,y'\in G$, there is always at least one \emph{shortest path} or \emph{geodesic path} $\gamma$ in $G$ joining $y,y'$ whose length satisfies $|\gamma|=d(y,y')$.

We define an undirected \emph{simple loop} $\mathcal{\gamma}$ of $G$ to be a simple path that intersects itself only at the endpoints.
We denote by $\mathcal{L}(G)$ the set of all simple loops of $G$.
Since $G$ is assumed to have finitely many edges, the set $\mathcal{L}(G)$ is finite.
Since a \emph{metric tree} is defined as a metric graph with no loops, $\mathcal{L}(G)=\emptyset$ for a metric tree $G$.

For a metric graph $G$ that is not an interval, 
let $e(G)$ denote the length of the shortest edge in $G$.

If a subset $A\subseteq G$ has the property that between any two points $a,a'\in A$ there exists a unique shortest geodesic in $G$, then we can define the \emph{(geodesic) convex hull} of $A$ in $G$ to be the union of geodesics between all $a$ and $a'$.
For example, if $\diam(A)<\frac{e(G)}{2}$, then the convex hull of $A$ in $G$ is well-defined.

Given a (connected and finite) metric graph $G$, let $G_0$ be the smallest connected subgraph of $G$ containing the union of all simple loops of $G$.
We call $G_0$ the {\em core of $G$}.
If $G$ has no leaf edges, then $G = G_0$.
Note that $G_0$ is just $G$ with its ``hanging trees'' removed, or in other words, the minimal subspace onto which $G$ deformation retracts.

We will use the following two properties about $G_0$.
First, if $G$ (and hence $G_0$) contains at least two simple loops, then at least one vertex of $G_0$ has degree at least $3$.
We can then consider the canonical graph representation of $G_0$ consisting of only vertices of degree at least $3$.
Second, we note that if a path connected subset $Z \subseteq G$ has first Betti  number equal to the first Betti number of $G$, then $G_0 \subseteq Z$.

\subsection{Hausdorff distances}
Let $(Z, d)$ be a metric space.
If $X$ and $Y$ are two subsets of $Z$, then the \emph{directed Hausdorff distance} from $X$ to $Y$, denoted $\vec{d}_\h(X,Y)$, is defined by
\[
\vec{d}_\mathrm{H}(X,Y)
=\sup_{x\in X}\inf_{y\in Y}d(x,y)
=\inf\bigg\{r\geq 0\mid X\subseteq Y^r\bigg\}
.
\]
Note that $\vec{d}_\h(X,Y)$ is not symmetric in general.

To retain symmetry, the \emph{Hausdorff distance}, denoted $d_\h(X,Y)$, is defined as
\[
d_\h(X,Y)=\max\left\{\vec{d}_\h(X,Y), \vec{d}_\h(X,Y)\right\}.
\]
In other words, the Hausdorff distance finds the infimum over all real numbers $r$ such that if we thicken $Y$ by $r$ it contains $X$, and if we thicken $X$ by $r$ it contains $Y$.
If no such $r$ exists then the Hausdorff distance is infinity.

\subsection{Gromov--Hausdorff distances}

Let $(X,d_X)$ and $(Y,d_Y)$ be metric spaces.
In equations~\eqref{eq:dgh-embeddings}--\eqref{eq:dgh-functions} we give three equivalent definitions of the Gromov--Hausdorff distance between $X$ and $Y$.

The first definition follows~\cite{gromov1981structures}.
The \emph{Gromov--Hausdorff distance $d_\gh(X,Y)$} is the infimum of the Hausdorff distances between $F(X)$ and $G(Y)$ in $Z$, 
\begin{equation}
\label{eq:dgh-embeddings}
d_\gh(X,Y)=\inf_{F,G,Z}d_\h(F(X),G(Y)),
\end{equation}
with the infimum being over all metric spaces $Z$ and all isometric embeddings $F\colon X \hookrightarrow Z $ and $G\colon Y \hookrightarrow Z$.

The second definition uses correspondences.
A relation $\cR\subseteq X\times Y$ is called a \emph{correspondence} if the following two conditions hold: for every $x \in X$ there exists some $y \in Y$ with $(x,y) \in \cR$, and
for every $y \in Y$ there exists some $x \in X$ exists with $(x,y) \in \cR$.
The \emph{distortion} of a correspondence $\cR$ is
\[\dis(\cR)=\sup_{(x,y),(x',y')\in \cR} \left|d_{X}(x,x')-d_{Y}(y,y')\right|.\]
The \emph{Gromov--Hausdorff} distance between two (arbitrary) metric spaces $X$ and $Y$ is
\begin{equation}
\label{eq:dgh-correspondences}
d_{\gh}(X,Y) = \tfrac{1}{2} \inf_{\cR \subseteq X \times Y}\dis(\cR),
\end{equation}
where the infimum is taken over all correspondences $\cR$ between $X$ and $Y$~\cite{BuragoBuragoIvanov,kalton1999distances}.
As shown in~\cite{ivanov2016realizations}, if $X$ and $Y$ are both compact metric spaces, a distortion-minimizing correspondence in the definition of $d_\gh(X,Y)$ always exists.

For metric spaces $(X,d_X)$ and $(Y,d_Y)$, let $f \colon X \to Y$ and $g \colon Y \to X$ be two (possibly discontinuous) functions between them.
We define the \emph{distortion of $f$} as
\[\dis(f)=\sup_{x,x'\in X}\left|d_{X}(x,x')-d_{Y}(f(x),f(x'))\right|.\]
The distortion of $g$ is defined similarly.
The \emph{codistortion of $f$ and $g$} is defined as
\[\codis(f,g)=\sup_{x\in X,y\in Y}\left|d_{X}(x,g(y))-d_{Y}(f(x),y)\right|.\]
Kalton and Ostrovskii~\cite{kalton1999distances} show that
\begin{equation}
\label{eq:dgh-functions}
d_{\gh}(X,Y) = \tfrac{1}{2}\inf_{f,g}\max\{\dis(f),\dis(g),\codis(f,g)\},
\end{equation}
where $f\colon X \to Y$ and $g\colon Y \to X$ are any functions.
It follows that if $d_\gh(X,Y) < C$, then there is some function $f\colon X\to Y$ with $\frac{1}{2}\dis(f) < C$.

\subsection{Gromov--Hausdorff distances and compactness}

The following result is folklore.

\begin{lemma}
\label{lemma:H_GH_and_compactness}
Let $X$ be a metric space and $Y$ be a subset thereof.
Let us denote by $\overline Y$ the closure of $Y$ in $X$.
Then
\[
d_\h(X,Y)=d_\h(X,\overline Y)\quad\text{and}\quad d_\gh(X,Y)=d_\gh(X,\overline Y).
\]
Furthermore, for every two subsets $Y$ and $Z$ of $X$, we have $\vec d_\h(Z,Y)=\vec d_\h(Z,\overline Y)$.
\end{lemma}

\begin{proof}
The first equality is straightforward, and the second one can be easily derived.
Indeed, using the triangle inequality of the Gromov--Hausdorff distance and the fact that the Hausdorff distance upper bounds the Gromov--Hausdorff distance, we obtain
\[
d_\gh(X,Y)\leq d_\gh(X,\overline Y)+d_\gh(\overline Y,Y)\leq d_\gh(X,\overline Y)+d_\h(\overline Y,Y)=d_\gh(X,\overline Y),
\]
and, similarly, $d_\gh(X,\overline Y)\leq d_\gh(X,Y)$.

Finally, we can use the triangle inequality of the directed Hausdorff distance and the first stated equality to show the final result:
\[
\vec d_\h(Z,Y)\leq\vec d_\h(Z,\overline Y)+\vec d_\h(\overline Y,Y)\leq\vec d_\h(Z,\overline Y)+d_\h(\overline Y,Y)=d_\h(Z,\overline Y),
\]
and similarly $\vec d_\h(Z,\overline Y)\leq d_\h(Z,Y)$.
\end{proof}

We will consistently apply Lemma~\ref{lemma:H_GH_and_compactness} to the case where the ambient metric space is compact.
Thus, when we are studying the Hausdorff and the Gromov--Hausdorff distances between a compact metric space and a subspace thereof, we can assume---replacing the subspace with its closure---that the subspace is compact as well.

\subsection{Gromov--Hausdorff distances and connectedness}

\begin{lemma}
\label{lemma:path_connected_distortion2}
Let $f\colon X\to Y$ be a function between metric spaces, with $X$ path-connected and $Y$ geodesic.
Then for every $r>\tfrac{\dis(f)}{2}$, the thickening $f(X)^r$ is path connected.
\end{lemma}

\begin{proof}
Let $y_1,y_2\in f(X)^r$ and $x_1,x_2\in X$ be points such that $d_Y(y_i,f(x_i))<r$ for $i=1,2$.
Since $Y$ is geodesic, the entire geodesic connecting $y_i$ to $f(x_i)$ is contained in $B(f(x_i),r)\subseteq f(X)^r$.
Pick $\varepsilon>0$ sufficiently small to satisfy $r>\tfrac{\dis(f)}{2}+\varepsilon$.
Since $X$ is path connected, there are $z_1=x_1,z_2,\dots,z_{n+1}=x_2$ such that $d(z_j,z_{j+1})<\varepsilon$ for every $j=0,\dots,n$.
Thus, $d(f(z_j),f(z_{j+1}))<\dis(f)+\varepsilon$.
Since $Y$ is geodesic, there are $m_j\in Y$ such that $d(f(z_j),m_j)=d(f(z_{j+1}),m_j)=\tfrac{1}{2}d(f(z_j),f(z_{j+1}))<r$ and the geodesics connecting $m_j$ to $f(z_j)$ and $f(z_{j+1})$ are contained in $f(X)^r$.
\end{proof}

\subsection{Gromov--Hausdorff distances and graphs}

Let $G$ be a finite metric graph and let $X\subseteq G$ be a compact subset.
To use functions $f\colon G\to X$ to prove bounds on the Gromov--Hausdorff distance, we need to be a bit careful.
Indeed, even if the sample $X$ is close to the leaves of $G$, it may be that $f(G)$ is not close to the leaves.
Luckily, we have a lot of freedom to construct our function $f$.
In the following lemma we show how this can be done.

\begin{lemma}
\label{lemma:map_close_to_the_leaves}
Let $G$ be a finite metric graph and let $X \subseteq G$ be compact.
If $d_\gh(G,X)<r$, then there is a function $f\colon G\to X$ with $\frac{1}{2}\dis(f)<r$ and $\vec d_\h(\partial G,f(G))=\vec d_\h(\partial G,X)$.
\end{lemma}

\begin{proof}
Let $\cR\subseteq G\times X$ be a correspondence satisfying $\dis(\cR)<2r$.
Fix $\varepsilon>0$ such that $\dis(\cR)+\varepsilon<2r$.
For every $v\in\partial G$, pick $x_v\in X$ as a point satisfying $d(v,x_v)=\min_{x\in X}d(v,x)$.
We first recursively construct a surjective map $\widetilde f$ from a subset of $G$ to $\{x_v\mid v\in\partial G\}$.
Let us arbitrarily order the leaves $v_1,\dots,v_n$ and denote the point $x_{v_i}$ by $x_i$ for the sake of simplicity.

{\bf Base step.}
Pick any $y_1\in G$ such that $(y_1,x_1)\in\cR$ and set $\widetilde f(y_1)=x_1$.

{\bf Recursive step.}
Assume we have defined $y_1,\dots,y_m\in G$ distinct and set $\widetilde f(y_i)=x_i$.
Pick $y\in G$ satisfying $(y,x_{m+1})\in\cR$.
Let us distinguish two cases.
If $y\notin\{y_1,\dots,y_m\}$, then set $y_{m+1}=y$ and $\widetilde f(y_{m+1})=x_{m+1}$.
Otherwise, fix an arbitrary point $y_{m+1}\in G\setminus\{y_1,\dots,y_m\}$ satisfying $d(y_{m+1},y)<\frac{\varepsilon}{2}$ and define $\widetilde f(y_{m+1})=x_{m+1}$.

Consider an arbitrary extension $f\colon G\to X$ of $\widetilde f\colon\{y_1,\dots,y_n\}\to\{x_1,\dots,x_n\}$ with the property that the graph of $f|_{G\setminus\{y_1,\dots,y_n\}}$ is contained in $\cR$.
Then $\dis(f)\leq\dis(\cR)+2\cdot \tfrac{\varepsilon}{2}<2r$ and $\vec d_\h(\partial G,f(G))=\vec d_\h(\partial G,X)$ by construction.
\end{proof}

\section{Ratio of $d_\gh$ vs $d_\h$ can be arbitrarily small for metric trees}
\label{sec:ratio}

In this section we prove Theorem~\ref{thm:counterExample}, a variant of Theorem~5 of~\cite{HvsGH}, now adapted to metric trees.
As we show in Theorems~\ref{thm:maps_and_trees} and~\ref{thm:graphs}, the Gromov--Hausdorff distance $d_\gh(T,X)$ between a metric graph $T$ and a subset $X\subseteq T$ can be bounded below by their Hausdorff distance $d_\h(T,X)$ when $X$ is close enough to the leaves of $T$, i.e., when $d_\h(T,X)>\vec{d}_\h(\partial T, X)$.
When $d_\h(T,X) < \vec{d}_\h(\partial T, X)$, however, the ratio of $d_\gh(T,X)$ over $d_\h(T,X)$ can become arbitrarily small as we show below.
Furthermore, Theorem~\ref{thm:counterExample} suggests that $\partial T$ may not map close to itself via correspondences with small distortion; see Figure~\ref{Fig:1}.

\begin{figure}[!ht]
\begin{center}
\begin{tikzpicture}[scale=1]
\foreach \x in {10, 20, ..., 170}
\draw (0:0)--(\x : 1 + .01 * \x);
\draw (0,-.5) node {$T$};
\end{tikzpicture}
\quad
\begin{tikzpicture}[scale=1]
\foreach \x in {10, 20, ..., 170}
\draw (0:0)--(\x : 1 + .01 * \x);
\foreach \x in {10, 20, ..., 170}
\draw[ultra thick] (0:0)--(\x : .9 + .01 * \x);
\draw (0,-.5) node {$X' \subseteq T$};
\end{tikzpicture}
\quad
\begin{tikzpicture}[scale=1]
\foreach \x in {10, 20, ..., 170}
\draw (0:0)--(\x : 1 + .01 * \x);
\foreach \x in {10,20, ..., 160}
\draw[ultra thick] (0:0)--(\x: 1 + .01 * \x);
\draw[ultra thick] (0:0)--(170 : 1);
\draw (0,-.5) node {$X \subseteq T$};
\end{tikzpicture}
\end{center}
\caption{A sketch of the construction of Theorem~\ref{thm:counterExample}: (left) The metric tree $T$, (middle) the bold superimposed $X' \subseteq T$ with small $d_\h(T,X')$, and (right) a different isometric embedding of $X'$ into $T$ indicated by $X \subseteq T$ with large $d_\h(T,X)$.}
\label{Fig:1}
\end{figure}

\begin{theorem}
\label{thm:counterExample}
For any $\varepsilon > 0$ there exists a finite metric tree $T$ and a closed subset $X\subseteq T$ such that $d_\gh(T, X) \leq \varepsilon \cdot d_\h(T, X)$.
\end{theorem}

\begin{proof}
Choose $n\in \mathbb{N}$ and let $\varepsilon = \frac{1}{n}$.
For $t \geq 0$, let $I(t)$ be a copy of the interval $[0,t]$ with a designated basepoint $0$.
Define a metric tree $T$ as the wedge of $n$ intervals of lengths $1+\varepsilon, 1+2\varepsilon, \ldots, 1+n\varepsilon=2$, and define a metric tree $X_0$ as the wedge of intervals of lengths $1,1+\varepsilon, 1+2\varepsilon, \ldots, 1+(n-1)\varepsilon=2-\varepsilon$:
\[T = \bigvee_{t=1}^n I(1+ t \cdot \varepsilon)
\quad\text{ and }\quad
X_0 = \bigvee_{t=0}^{n-1} I(1+t \cdot \varepsilon).\]

Recall from~\eqref{eq:dgh-embeddings} that $d_\gh(T,X)$ is the infimum of the Hausdorff distances between $F(T)$ and $G(X)$ in $Z$, with the infimum being over all metric spaces $Z$ and all isometric embeddings $F\colon T \hookrightarrow Z $ and $G\colon X \hookrightarrow Z$.

In Figure~\ref{Fig:1} we have two different isometric embeddings $X', X$ of the same tree $X_0$ into $Z=T$:
one with a small Hausdorff distance $d_\h(T,X')=\varepsilon$ as $X'$ is $T$ with each ray being shortened by $\varepsilon$, and one with a large Hausdorff distance $d_\h(T,X)=1$ as $X$ is $T$ with the longest ray $I(2)$ of length $2$ being replaced with a subray $I(1)$ of length $1$.
By~\eqref{eq:dgh-embeddings}, we have $d_\gh(T,X) \leq d_\h(T,X') = \varepsilon d_\h(T,X)$.
\end{proof}

An analogy can be made with the construction in Theorem~5 of~\cite{HvsGH}, where instead of permuting standard basis vectors in Euclidean space, we are now permuting the leaves of a metric tree.

\section{Metric trees with coefficient $1$}
\label{sec:trees}

Section~\ref{sec:ratio} demonstrates that results claiming $d_\gh(G,X)=d_\h(G,X)$ for $X$ a subset of a metric graph $G$ do not hold in general, even if $d_\h(G,X)$ is arbitrarily small.
In this section, as well as in Section~\ref{sec:graphs}, we prove related results relying on bounds on $\vec{d}_\h(\partial G,X)$.
In particular, we will assume the condition  $d_\h(G,X)> \vec{d}_\h(\partial G,X)$ which rules out the example constructed in Theorem~\ref{thm:counterExample}.
This condition has a high probability of being satisfied if $X$ is a large uniform sample of $G$.

\begin{proposition}
\label{prop:assumption_on_leaves_generic}
If $G$ is a metric graph and $X$ is a sample of $n$ uniformly random points, then with high probability as $n \to \infty$, we have $d_\h(G,X) > \vec{d}_\h(\partial G,X)$.
\end{proposition}

\begin{proof}
Normalize so that the total length of $G$ is $1$.
We show that with high probability,
\[d_\h(G,X)>\tfrac{\sqrt{\log n}}{n}>\vec{d}_\h(\partial G,X).\]
Note that for any region of length $\ell$, the probability that $X$ misses that region is
$\left(1-\ell\right)^n \approx e^{-\ell n}$.

We first show the second inequality.
Let $k$ be the number of leaves in the graph.
Then $\vec{d}_\h(\partial G,X)\ge\frac{\sqrt{\log n}}{n}$ if and only if $X$ is disjoint from the $\frac{\sqrt{\log n}}{n}$-neighborhood of the leaves.
For large enough $n$, the length of this neighborhood is $\frac{k\sqrt{\log n}}{n}$, so the probability that $X$ completely misses this region is
\[\left(1-\tfrac{k\sqrt{\log n}}{n}\right)^n \approx e^{-k\sqrt{\log n}} \to 0.\]

Now we show the first inequality.
Let $\{y_i\}$ be a maximal set of points in $G$ whose $\tfrac{\sqrt{\log n}}{n}$-neighborhoods $B_i$ don't overlap and don't hit a vertex; note the cardinality of this set is $r=\Theta\big(\tfrac{n}{\sqrt{\log n}}\big)$.
We show the first inequality by showing that with high probability, $X$ misses at least one of the $B_i$.
If the events that $X$ misses $B_i$, for each $i$, were independent, then we would have
\begin{align*}
\mathbb{P}[X\text{ hits all of the }B_i] &= \mathbb{P}[X\text{ hits a given }B_i]^{\Theta\bigl(\frac{n}{\sqrt{\log n}}\bigr)} \approx \biggl(1-e^{-2\sqrt{\log n}}\biggr)^{\Theta\bigl(\frac{n}{\sqrt{\log n}}\bigr)} \to 0.
\end{align*}
In fact, they are negatively correlated, so the equality above becomes $\leq$.
To make this precise, for any $0\leq j\leq n$,  let $p_j$ be the probability that exactly $j$ sample points hit $\bigcup_i B_i$.
Then, thinking of the number of times each $B_i$ is hit as defining a partition of $\{1,\ldots,j\}$ into $r$ subsets, we get
\[
\mathbb{P}(X\text{ hits all of the }B_i\mid j\text{ sample points hit }{\textstyle\bigcup_i B_i})
=\tfrac{{j-1\choose r-1}}{{j+r-1\choose r-1}}.
\]
Now we have
\begin{align*}
\mathbb{P}(X\text{ hits all of the }B_i)
&=\sum_{j=0}^n p_j
\mathbb{P}(X\text{ hits all of the }B_i\mid j\text{ sample points hit }{\textstyle\bigcup_i B_i})\\
&=\sum_{j=0}^n p_j\tfrac{{j-1\choose r-1}}{{j+r-1\choose r-1}}
=\sum_{j=0}^n p_j\left[\tfrac{j-1}{j+r-1}\cdot\tfrac{j-2}{j+r-2}\ldots\tfrac{j-r+1}{j+1}\right]\\
&\leq\sum_{j=0}^n p_j\left(\tfrac{n-1}{n+r-1}\right)^r
=\left(\tfrac{n-1}{n+r-1}\right)^r\sum_{j=0}^n p_j\\
&=\left(1-\tfrac{r}{n+r-1}\right)^r\quad\quad\quad\text{since }\textstyle{\sum_{j=0}^n p_j=1} \\
&=\left(1-\tfrac{r/n}{1+r/n-1/n}\right)^r \\
&=\left(1-\tfrac{\Theta(1/\sqrt{\log{n}})}{1+\Theta(1/\sqrt{\log{n}})-1/n}\right)^{\Theta(n/\sqrt{\log{n}})}\xrightarrow{n\to\infty}0.
\end{align*}
It follows that with high probability, some $y_i$ is at distance at least $\frac{\sqrt{\log n}}{n}$ from $X$.
\end{proof}

We obtain the following results for metric trees.

\begin{proposition}
\label{thm:tree_via_maps}
Let $T$ be a finite tree, and let $f\colon T\to T$ be a function.
If $d_\h(T,f(T))>\vec d_\h(\partial T,f(T))$, then $\tfrac{1}{2}\dis(f)\geq d_\h(T,f(T))$.
\end{proposition}

\begin{proof}
Let $\tfrac{1}{2}\dis(f)<r$ and let $\overline x\in T$ be a point such that $B(\overline x,d_\h(T,f(T)))\cap f(T)=\emptyset$.
The assumption $d_\h(T,f(T))>\vec d_\h(\partial T,f(T))$ implies that $T\setminus\{\overline x\}$ is disconnected and at least two distinct components thereof have non-empty intersection with $f(T)$ (since a tree has at least two leaves).
Since $f(T)^r$ is connected by Lemma~\ref{lemma:path_connected_distortion2}, we have $\overline x\in f(T)^r$, which implies $d_\h(T,f(T))<r$.
Since $r$ can be taken arbitrarily, $d_\h(T,f(T))\leq\tfrac{1}{2}\dis(f)$.
\end{proof}

\begin{theorem}
\label{thm:maps_and_trees}
Let $T$ be a finite metric tree and let $X\subseteq T$.
If $d_\h(T,X)>\vec d_\h(\partial T,X)$, then $d_\gh(T,X)=d_\h(T,X)$.
\end{theorem}

\begin{proof}
By Lemma~\ref{lemma:H_GH_and_compactness}, we may assume that $X$ is compact.
Indeed, we can consider $\overline X$, which is compact since $T$ is.
Note $\overline X$ still satisfies the same hypothesis as $X$ since 
\[
d_\h(T,\overline X)=d_\h(T,X)>\vec d_\h(\partial T,X)=\vec d_\h(\partial T,\overline X).
\]
Finally, once we achieve the equality $d_\gh(T,\overline X)=d_\h(T,\overline X)$, the corresponding claimed equality for $X$ follows given that $d_\gh(T,\overline X)=d_\gh(T,X)$.

Suppose for a contradiction that $d_\gh(T,X)<d_\h(T,X)$.
Then by Lemma~\ref{lemma:map_close_to_the_leaves}, there is a function $f\colon T\to X\subseteq T$ with $\tfrac{1}{2}\dis(f)< d_\h(T,X)\leq d_\h(T,f(T))$ and 
\[\vec d_\h(\partial T,f(T))=\vec d_\h(\partial T,X)<d_\h(T,X)\leq d_\h(T,f(T)),\]
contradicting Proposition~\ref{thm:tree_via_maps}.
\end{proof}

In other words, $d_\h(T,X) =\max\{d_\gh(T,X), \vec{d}_\h(\partial T, X)\}$.

We remark that if $d_\h(T,X)=\vec{d}_\h(\partial T,X)$, then $d_\gh(T,X)$ can be arbitrarily small with respect to $d_\h(T,X)$ as demonstrated in Theorem~\ref{thm:counterExample}.

\begin{remark}
We can use the proof strategy of Theorem~\ref{thm:maps_and_trees} in a wider context.
In particular, it can be applied even if the sample is not close to all the leaves.
It is enough that the empty ball realizing the Hausdorff distance between $G$ and $X$ disconnects $X$.
Consider the tree $T$ and the subset $X$ as represented in the following figure:

\begin{figure}[h!]
\centering
\begin{tikzpicture}
\draw (0,0)--(6,0) (3.5,0)--(3.5,0.3);
\draw (2.5,0.5) node{$T$};
\fill (3,0) circle (2pt) (3.5,0.3) circle (2pt);
\draw (3,0) node[below]{$\overline x$};
\draw (3.5,0.3) node[right]{$v$};
\draw[red,ultra thick] (0,0)--(2,0) (4,0)--(6,0);
\draw[red] (1,0) node[above]{$X$};
\end{tikzpicture}
\label{fig:density_around_the_leaves_too_strong}
\end{figure}
The point $\overline x\in T$ satisfies $B(\overline x,d_\h(G,X))\cap X=\emptyset$ and the ball contains the leaf $v$.
Despite this, we can still apply the proof strategy, with straightforward modifications, because $X\cap (T\setminus\{\overline x\})$ has at least two connected components, concluding that $d_\gh(T,X)=d_\h(T,X)$.
\end{remark}

All of the results in our paper have simple corollaries for two subsets of the tree or graph; we give one such example.

\begin{corollary}
Let $T$ be a finite metric tree and let $X,Y\subseteq T$.
For any $\varepsilon\ge d_\h(T,Y)$, if $d_\h(X,Y)>\vec{d}_\h(\partial T, X)+\varepsilon$, then $d_\gh(X,Y)\ge d_\h(X,Y)-2\varepsilon$.
\end{corollary}

\begin{proof}
We need only apply the triangle inequality:
\begin{align*}
d_\h(X,Y) &\le d_\h(T,X) + d_\h(T,Y) \\
&=\max\left\{d_\gh(T,X), \vec{d}_\h(\partial T, X)\right\} + d_\h(T,Y) && \text{by Theorem~\ref{thm:maps_and_trees}}\\
&\le \max\left\{d_\gh(X, Y)+ d_\gh(T,Y), \vec{d}_\h(\partial T, X)\right\} + d_\h(T,Y) \\
&\le \max\left\{d_\gh(X, Y)+ d_\h(T,Y), \vec{d}_\h(\partial T, X)\right\} + d_\h(T,Y) && \text{since }d_\gh(T,Y)\leq d_\h(T,Y) \\
&= \max\left\{d_\gh(X, Y)+ 2d_\h(T,Y), \vec{d}_\h(\partial T, X) + d_\h(T,Y)\right\} \\
&\le \max\left\{d_\gh(X, Y)+ 2\varepsilon, \vec{d}_\h(\partial T, X) + \varepsilon\right\}
&& \text{since }d_\h(T,Y)\leq\varepsilon.
\qedhere
\end{align*}
\end{proof}

As a particular case of Theorem~\ref{thm:maps_and_trees}, we can provide a closed formula to compute the Gromov--Hausdorff distance between a segment and a sample of it.

\begin{corollary}
\label{coro:segment}
Let $X$ be a compact non-empty subset of an interval $I=[a,b]$.
Denote by $c=\min X$ and $d=\max X$.
Then $d_\gh(I,X)=\max\left\{\frac{c-a+b-d}{2},d_\h([c,d],X)\right\}$.
\end{corollary}

\begin{proof}
Let $X'\subseteq I$ be an isometric copy of $X$ centered in the interval $I$, meaning the distance from $a$ to the closest point in $X'$ is equal to the distance from $b$ to the closest point in $X'$; see Figure~\ref{fig:interval}.
Then 
\[d_{\gh}(I,X)=d_\gh(I,X')\leq d_\h(I,X')=\max\Big\{\tfrac{c-a+b-d}{2},d_\h([c,d],X)\Big\}.\]

We have $d_\gh(I,X)\geq\frac{\diam(I)-\diam(X)}{2}=\frac{c-a+b-d}{2}$.
There are now two cases.
If $d_\h([c,d],X)\leq \frac{c-a+b-d}{2}$, the claim holds.
Alternatively, suppose that $d_\h([c,d],X)>\frac{c-a+b-d}{2}$.
Then, this assumption implies that $d_\h(I,X')=d_\h([c,d],X)$, and also that $d_\h(I,X')>\vec{d}_H(\partial I, X')$ since $\tfrac{c-a+b-d}{2}=\vec{d}_H(\partial I, X')$.
In virtue of the latter inequality, an application of Theorem~\ref{thm:maps_and_trees} to the pair $X'\subseteq I$ gives $d_\h(I,X')=d_\gh(I,X')$, and hence
\[ d_\h([c,d],X) = d_\h(I,X') = d_\gh(I,X') = d_\gh(I,X).\]
\end{proof}

\begin{figure}[hbt]
\centering
\begin{subfigure}{0.45\textwidth}
\centering
\begin{tikzpicture}
\draw (0,0)--(5,0);
\draw (0,0) node[above]{$a$};
\draw (5,0) node[above]{$b$};
\draw[ultra thick,red] (0.5,0)--(0.6,0) (0.8,0)--(1.2,0) (2.5,0)--(2.6,0) (2.9,0)--(3.5,0) (1.3,0)--(1.5,0);
\draw[red] (0.5,0) node[below]{$c$};
\draw[red] (3.5,0) node[below]{$d$};
\draw[red] (0.5,0.2)--(0.5,0.4)--(3.5,0.4)node[above,pos=0.5]{$X$}--(3.5,0.2);
 \draw[white] (0,-0.2)--(0,-0.4)--(1,-0.4)node[below,pos=0.5]{$\tfrac{c-a+b-d}{2}$}--(1,-0.2);
\draw[white] (4,-0.2)--(4,-0.4)--(5,-0.4)node[below,pos=0.5]{$\tfrac{c-a+b-d}{2}$}--(5,-0.2);
\draw (4.2,0) node[below]{$I$};
\end{tikzpicture}
\end{subfigure}
\begin{subfigure}{0.45\textwidth}
\centering
\begin{tikzpicture}
\draw (0,0)--(5,0);
\draw (0,0) node[above]{$a$};
\draw (5,0) node[above]{$b$};
\draw[ultra thick,red] (1,0)--(1.1,0) (1.3,0)--(1.7,0) (3,0)--(3.1,0) (3.4,0)--(4,0) (1.8,0)--(2,0);
\draw[red] (1,0.2)--(1,0.4)--(4,0.4)node[above,pos=0.5]{$X^\prime$}--(4,0.2);
\draw (0,-0.2)--(0,-0.4)--(1,-0.4)node[below,pos=0.5]{$\tfrac{c-a+b-d}{2}$}--(1,-0.2);
\draw (4,-0.2)--(4,-0.4)--(5,-0.4)node[below,pos=0.5]{$\tfrac{c-a+b-d}{2}$}--(5,-0.2);
\end{tikzpicture}
\end{subfigure}
\caption{A representation of the interval $I=[a,b]$ and the subsets $X$ and $X^\prime$ as described in Corollary~\ref{coro:segment} and its proof.}
\label{fig:interval}
\end{figure}

\section{The case of the circle}
\label{sec:circle}

Before considering general metric graphs in the following section, in this section we first understand the simplest connected metric graph which is not a tree: the circle.

In the case of the circle ($G=S^1$), the authors of~\cite{HvsGH} showed that $d_\gh(S^1,X)=d_\h(S^1,X)$ for a compact subset $X\subseteq S^1$ with $d_\gh(S^1,X)<\frac{\pi}{6}$.
A curious question regarding the optimality of the density constant $\frac{\pi}{6}$ was also posed therein~\cite[Question~2]{HvsGH}.
While the constant $\frac{\pi}{6}$ provides a sufficient Gromov--Hausdorff density of a subset $X$ to guarantee the equality of its Hausdorff and Gromov--Hausdorff distances to $S^1$, it turns out this constant is suboptimal.
In this section, we show that $\frac{\pi}{3}$ is the optimal constant in the case of the circle.
That is, for a subset $X\subseteq S^1$, the density condition $d_\gh(S^1,X)<\frac{\pi}{3}$ implies $d_\gh(S^1,X)=d_\h(S^1,X)$ (Theorem~\ref{thm:circleSufficient}), and there is a compact $X\subseteq S^1$ with $\frac{\pi}{3}=d_\gh(S^1,X)<d_\h(S^1,X)$ (Theorem~\ref{thm:circleOptimal}).

\begin{theorem}
\label{thm:circleSufficient}
For any subset $X \subseteq S^1$, we have $d_\gh(S^1, X) \ge \min\left\{ d_\h(S^1, X),\ \tfrac{\pi}{3} \right\}$.
\end{theorem}

Our first proof of this theorem, available in version~1 on arXiv~\cite{adams2024lower}, was longer but relied on even more elementary topological tools, not even the Borsuk--Ulam theorem.

\begin{proof}
This proof is closely related to~\cite{leon,dubins1981equidiscontinuity}.
It suffices to show that if $d_\gh(S^1,X)<\frac{\pi}{3}$, then $d_\gh(S^1,X)\ge d_\h(S^1,X)$.

Let $f\colon S^1\to X$ be a (possibly discontinuous) function.
By~\eqref{eq:dgh-functions}, it suffices to show that if $\frac{1}{2}\dis(f)<\frac{\pi}{3}$, then $\frac{1}{2}\dis(f)\ge d_\h(S^1,X)$, or equivalently that $\frac{\varepsilon+\dis(f)}{2} \ge d_\h(S^1,X)$ for all $\varepsilon>0$ sufficiently small.

Suppose $\frac{1}{2}\dis(f)<\frac{\pi}{3}$.
Fix $\varepsilon>0$ with $5\varepsilon<2\pi-3\dis(f)$.
Form a finite triangulation $K$ of $S^1$ with edges of length less than $\varepsilon$.
Let $K^{(0)}\subseteq S^1$ be the vertex set of $K$.
Since adjacent vertices $v_i,v_{i+1}\in K^{(0)}$ satisfy $d(v_i,v_{i+1})<\varepsilon$, we have $d(f(v_i),f(v_{i+1}))<\varepsilon+\dis(f)<\pi$.
Hence we can map the edge between $v_i$ and $v_{i+1}$ to the unique shortest path in $S^1$ between $f(v_i)$ and $f(v_{i+1})$, obtaining a continuous map $\overline{f}\colon K\to S^1$.

\begin{figure}[h!]
\centering
\begin{tikzpicture}
\draw[line width=8,red,opacity=0.2,domain=15:130] plot ({2*cos(\x)+6.5}, {2*sin(\x)});
\draw[line width=8,red,opacity=0.2,domain=70:170] plot ({2*cos(\x)+6.5}, {2*sin(\x)});
\draw[thick] (0,0) circle (2cm) (6.5,0) circle (2cm);
\draw[->] (2.5,0)--(4,0)node[pos=0.5,above]{$f$};
\fill (canvas polar cs:radius=2cm,angle=30) circle (2pt) (canvas polar cs:radius=2cm,angle=45) circle (2pt) (canvas polar cs:radius=2cm,angle=210) circle (2pt) (canvas polar cs:radius=2cm,angle=225) circle (2pt);
\fill[blue] (canvas polar cs:radius=2cm,angle=40) circle (2pt) (canvas polar cs:radius=2cm,angle=220) circle (2pt);
\draw[blue] (canvas polar cs:radius=2.4cm,angle=40) node{$x$};
\draw[blue] (canvas polar cs:radius=2.4cm,angle=220) node{$-x$};
\draw (canvas polar cs:radius=1.6cm,angle=30) node{$v_i$};
\draw (canvas polar cs:radius=1.8cm,angle=45) node[left]{$v_{i+1}$};
\draw (canvas polar cs:radius=1.6cm,angle=210)node{$v_j$};
\draw (canvas polar cs:radius=1.8cm,angle=225)node[right]{$v_{j+1}$};
\fill (6.5,0)+(canvas polar cs:radius=2cm,angle=15) circle (2pt)
(6.5,0)+(canvas polar cs:radius=2cm,angle=130) circle (2pt)
(6.5,0)+(canvas polar cs:radius=2cm,angle=70) circle (2pt)
(6.5,0)+(canvas polar cs:radius=2cm,angle=170) circle (2pt);
\fill[blue] (6.5,0)+(canvas polar cs:radius=2cm,angle=95) circle (2pt);
\draw[blue] (6.5,0)+(canvas polar cs:radius=2.4cm,angle=95) node{$f(x)=f(-x)$};
\draw (6.5,0)+(canvas polar cs:radius=1.4cm,angle=15) node{$f(v_j)$};
\draw (6.5,0)+(canvas polar cs:radius=1.4cm,angle=130) node{$f(v_{j+1})$};
\draw (6.5,0)+(canvas polar cs:radius=1.4cm,angle=170) node{$f(v_i)$};
\draw (6.5,0)+(canvas polar cs:radius=1.4cm,angle=70) node{$f(v_{i+1})$};
\end{tikzpicture}
\caption{
Proof of Theorem~\ref{thm:circleSufficient}: Since $\overline{f}(x)=\overline{f}(-x)$, the shortest paths in $S^1$ between $f(v_i), f(v_{i+1})$ and $f(v_j), f(v_{j+1})$ intersect.
}
\label{fig:Circle-LoopLifting}
\end{figure}

Let $r=\frac{\varepsilon+\dis(f)}{2}$.
Suppose for a contradiction that $X^r$ is not equal to all of $S^1$.
Then there is an arc of length at least $2r$ in $S^1$ that contains no points from $X$.
So $\overline{f}$ is non-surjective, since the image of each edge in $K$ is an edge of length less than $2r$.
Hence $\overline{f}\colon K\to S^1$ is a map of degree zero.
The Borsuk--Ulam theorem~\cite[Page~206, Exercise~21]{armstrong2013basic} then implies there are antipodal points $x,-x\in K$ with $\overline{f}(x)=\overline{f}(-x)$.
Let $x$ be on the closed edge between $v_i, v_{i+1}\in K^{(0)}$, and let $-x$ be on the closed edge between $v_j, v_{j+1}\in K^{(0)}$.
Since $\overline f(x)=\overline f(-x)$, the shortest path in $S^1$ between $f(v_i), f(v_{i+1})$ intersects the shortest path between $f(v_j), f(v_{j+1})$.
Since these paths in $S^1$ are of length less than $2r$, one of $f(v_i), f(v_{i+1})$ is at a distance less than $r$ from one of $f(v_j), f(v_{j+1})$, i.e.\ $d(f(v),f(v'))<r$ for some $v\in\{v_i,v_{i+1}\}$ and $v'\in\{v_j,v_{j+1}\}$; see Figure~\ref{fig:Circle-LoopLifting}.
We obtain
\[\dis(f)>d(v,v')-d(f(v),f(v'))>(\pi-2\varepsilon)-\tfrac{\varepsilon+\dis(f)}{2}=\pi-\tfrac{\dis(f)}{2}-\tfrac{5\varepsilon}{2}.\]
So $\tfrac{5\varepsilon}{2}>\pi-\frac{3\dis(f)}{2}$, contradicting the choice of $\varepsilon$.
Hence it must be that $X^r=S^1$.

We have $\frac{\varepsilon+\dis(f)}{2}= r\geq d_\h(S^1,X)$ for all $\varepsilon>0$ sufficiently small, completing the proof.
\end{proof}

\begin{figure}[htb]
\centering
\begin{subfigure}{0.45\textwidth}
\centering
\begin{tikzpicture}[scale=1]
\draw[dashed, lightgray] (-2, 0) -- (2, 0);
\draw[dashed, lightgray] (0, 2) -- (0, -2);
\draw[lightgray] (0, 0) circle (2);
\draw[very thick, blue, |-|] (-140:2) arc (-140:-220:2) node[near end,auto] {$\cR^{-1}[a]$};
\draw[very thick, red, |-|] (40:2) arc (40:-40:2) node[near start,auto] {$\cR^{-1}[b]$};
\draw[very thick, black, |-|] (-220:2) arc (-220:-270:2) node[near end,auto] {$\cR^{-1}[c]$};
\draw[very thick, green, |-|] (90:2) arc (90:40:2) node[near start,auto] {$\cR^{-1}[d]$};
\draw[very thick, magenta, |-|] (-40:2) arc (-40:-90:2) node[near end,auto] {$\cR^{-1}[f]$};
\draw[very thick, cyan, |-|] (-90:2) arc (-90:-140:2) node[near start,auto] {$\cR^{-1}[e]$};
\node[fill=white] at (0:2) {\small$\frac{2\pi}{3}-2\varepsilon$};
\node[fill=white] at (62:2) {\small$\frac{\pi}{6}+\varepsilon$};
\node[fill=white] at (-62:2) {\small$\frac{\pi}{6}+\varepsilon$};
\node[fill=white] at (118:2) {\small$\frac{\pi}{6}+\varepsilon$};
\node[fill=white] at (-118:2) {\small$\frac{\pi}{6}+\varepsilon$};
\node[fill=white] at (180:2) {\small$\frac{2\pi}{3}-2\varepsilon$};
\end{tikzpicture}
\subcaption{Correspondence $\cR\subseteq S^1\times X$ has a distortion of $\frac{2\pi}{3}$.
So, $d_\gh(S^1,X)\leq\frac{\pi}{3}$.}
\end{subfigure}\hfill
\begin{subfigure}{0.45\textwidth}
\centering
\begin{tikzpicture}[scale=1.1]
\draw[dashed, lightgray] (-2, 0) -- (2, 0);
\draw[dashed, lightgray] (0, 2) -- (0, -2);
\draw[gray] (0, 0) circle (2);
\fill[green] (0:2) circle (2pt) node[anchor=west] {$d$};
\fill[red] (-30:2) circle (2pt) node[anchor=north west] {$b$};
\fill[blue] (-150:2) circle (2pt) node[anchor=north east] {$a$};
\fill[cyan] (-180:2) circle (2pt) node[anchor=east] {$e$};
\fill[magenta] (50:2) circle (2pt) node[anchor=south west] {$f$};
\fill[black] (-230:2) circle (2pt) node[anchor=south east] {$c$};
\node[fill=white] at (-15:2) {\small$\frac{\pi}{6}-\varepsilon$};
\node[fill=white] at (-90:2) {\small$\frac{2\pi}{3}+2\varepsilon$};
\node[fill=white] at (-165:2) {\small$\frac{\pi}{6}-\varepsilon$};
\node[fill=white] at (25:2) {\small$\frac{\pi}{3}$};
\node[fill=white] at (90:2) {\small$\frac{\pi}{3}$};
\node[fill=white] at (155:2) {\small$\frac{\pi}{3}$};
\end{tikzpicture}
\subcaption{$X=\{a, b, c, d, e, f\}\subseteq S^1$ \\ with $d_\h(S^1,X)=\frac{\pi}{3}+\varepsilon$.}
\end{subfigure}
\caption{
The configuration of points $X\subseteq S^1$ is depicted to show the optimal constant for the circle.
}
\label{fig:non-optimal-1}
\end{figure}

We now prove the optimality of $\frac{\pi}{3}$ for the circle, i.e.\ that $\frac{\pi}{3}$ is also a necessary density condition.

\begin{theorem}
\label{thm:circleOptimal}
For any $\varepsilon\in\left(0,\frac{\pi}{6}\right)$, there exists a nonempty compact subset $X\subseteq S^1$ with $d_\h(S^1, X)=\frac{\pi}{3}+\varepsilon$ and $d_\gh(S^1, X)=\frac{\pi}{3}<d_\h(S^1, X)$.
\end{theorem}

\begin{proof}
We construct $X$ to be a six-element subset of $S^1$.
Both $X$ and a particular correspondence $\cR\subseteq S^1 \times X$ are shown in Figure~\ref{fig:non-optimal-1}.
One can check that $d_\h(S^1,X)=\frac{\pi}{3}+\varepsilon$.
Also, one can check that $\dis(\cR)=\frac{2\pi}{3}$, giving $d_\gh(S^1,X)\leq\frac{\pi}{3}$.
So $d_\gh(S^1,X)<d_\h(S^1,X)$.
Moreover, Theorem~\ref{thm:circleSufficient} implies that $d_\gh(S^1,X)=\frac{\pi}{3}$.
\end{proof}

\section{Metric graphs with loops with coefficient $1$}
\label{sec:graphs}

In this section, we extend our results for the circle to general metric graphs; see Theorem~\ref{thm:graphs}.

\subsection{Distortion of continuous maps between graphs}

We first need the following Borsuk--Ulam theorem for maps from the circle into trees.

\begin{lemma}
\label{lem:BU-circle-tree}
If $T$ is a tree and $f\colon S^1\to T$ is continuous, then there exist $x,x'\in S^1$ with $d(x,x')\ge\frac{2\pi}{3}$ and $f(x)=f(x')$.
\end{lemma}

\begin{proof}
Let $x_0,x_1,x_2\in S^1$ with $d(x_i,x_j)=\frac{2\pi}{3}$ for all $i,j\in\{0,1,2\}$ with $i\neq j$.
There are two cases.

\begin{figure}[htb]
    \centering
    \begin{tikzpicture}
\draw[thick] (0,0) circle (1.5cm);
\fill (canvas polar cs:radius=1.5cm,angle=30) circle (2pt) (canvas polar cs:radius=1.5cm,angle=270) circle (2pt) (canvas polar cs:radius=1.5cm,angle=150) circle (2pt);
\draw (canvas polar cs:radius=1.8cm,angle=30) node{$x_0$};
\draw (canvas polar cs:radius=1.8cm,angle=150) node{$x_1$};
\draw (canvas polar cs:radius=1.8cm,angle=270) node{$x_2$};
\draw [<->,domain=40:140] plot ({1.8*cos(\x)}, {1.8*sin(\x)});
\draw [<->,domain=160:260] plot ({1.8*cos(\x)}, {1.8*sin(\x)});
\draw [<->,domain=280:380] plot ({1.8*cos(\x)}, {1.8*sin(\x)});
\draw (0,1.8) node[above]{$\tfrac{2\pi}{3}$};
\draw (canvas polar cs:radius=2.1cm,angle=210) node{$\tfrac{2\pi}{3}$};
\draw (canvas polar cs:radius=2.1cm,angle=330) node{$\tfrac{2\pi}{3}$};
\draw[->] (2.5,0)--(4,0)node[pos=0.5,above]{$f$};
\draw[thick] (4.5,1.5)--(5.5,0)--(4.5,-1.5) (5.5,0)--(7.5,0)--(8.5,1.5) (7.5,0)--(8.5,-1.5);
\fill (5,-0.75) circle (2pt) node[right]{$f(x_{i+2})$};
\fill (6.8,0) circle (2pt) node[above]{$f(x_i)$};
\fill (8,0.75) circle (2pt) node[right]{$f(x_{i+1})$};
\end{tikzpicture}
\caption{Proof of Lemma~\ref{lem:BU-circle-tree}: The case where some $f(x_i)$ is on the unique shortest path in $T$ between $f(x_{i+1})$ and $f(x_{i+2})$.}
\label{fig:BU-circle-to-tree-1}
\end{figure}

If some $f(x_i)$ is on the unique shortest path in $T$ between $f(x_{i+1})$ and $f(x_{i+2})$ (all indices taken modulo 3), then by the intermediate value theorem there is some $y\in [x_{i+1},x_{i+2}]\subseteq S^1$ with $f(y)=f(x_i)$; see Figure~\ref{fig:BU-circle-to-tree-1}.
Note $d(x_i,y)\ge \frac{2\pi}{3}$, as required.

\begin{figure}[htb]
    \centering
    \begin{tikzpicture}
\draw[thick] (0,0) circle (1.5cm);
\fill (canvas polar cs:radius=1.5cm,angle=30) circle (2pt) (canvas polar cs:radius=1.5cm,angle=270) circle (2pt) (canvas polar cs:radius=1.5cm,angle=150) circle (2pt);
\fill[red] (canvas polar cs:radius=1.5cm,angle=75) circle (2pt) (canvas polar cs:radius=1.5cm,angle=220) circle (2pt) (canvas polar cs:radius=1.5cm,angle=320) circle (2pt);
\draw[red] (canvas polar cs:radius=1.1cm,angle=75) node{$y_0$};
\draw[red] (canvas polar cs:radius=1.1cm,angle=220) node{$y_1$};
\draw[red](canvas polar cs:radius=1.1cm,angle=320) node{$y_2$};
\draw (canvas polar cs:radius=1.8cm,angle=30) node{$x_0$};
\draw (canvas polar cs:radius=1.8cm,angle=150) node{$x_1$};
\draw (canvas polar cs:radius=1.8cm,angle=270) node{$x_2$};
\draw[thick] (3,1.5)--(4,0)--(3,-1.5) (4,0)--(6,0)--(7,1.5) (6,0)--(7,-1.5);
\fill (3.5,-0.75) circle (2pt) node[right]{$f(x_0)$};
\fill (6.5,0.75) circle (2pt) node[right]{$f(x_1)$};
\fill[red] (6,0) circle (2pt) node[below left]{$b$};
\fill (6.5,-0.75) circle (2pt) node[right]{$f(x_2)$};

\draw[thick] (10,0) circle (1.5cm);
\fill (10,0)+(canvas polar cs:radius=1.5cm,angle=30) circle (2pt) (10,0)+(canvas polar cs:radius=1.5cm,angle=270) circle (2pt) (10,0)+(canvas polar cs:radius=1.5cm,angle=150) circle (2pt);
\draw (10,0)+(canvas polar cs:radius=1.8cm,angle=30) node{$x_0$};
\draw (10,0)+(canvas polar cs:radius=1.8cm,angle=150) node{$x_1$};
\draw (10,0)+(canvas polar cs:radius=1.8cm,angle=270) node{$x_2$};
\fill[red] (10,0)+(canvas polar cs:radius=1.5cm,angle=75) circle (2pt) (10,0)+(canvas polar cs:radius=1.5cm,angle=10) circle (2pt) (10,0)+(canvas polar cs:radius=1.5cm,angle=165) circle (2pt);
\draw[red] (10,0)+(canvas polar cs:radius=1.1cm,angle=75) node{$y_0$};
\draw[red] (10,0)+(canvas polar cs:radius=1.1cm,angle=165) node{$y_1$};
\draw[red] (10,0)+(canvas polar cs:radius=1.1cm,angle=10) node{$y_2$};
\end{tikzpicture}
\caption{Proof of Lemma~\ref{lem:BU-circle-tree}: The case where no $f(x_i)$ is on the unique shortest path in $T$ between $f(x_{i+1})$ and $f(x_{i+2})$.}
\label{fig:BU-circle-to-tree-2}
\end{figure}

Alternatively, suppose no $f(x_i)$ is on the unique shortest path in $T$ between $f(x_{i+1})$ and $f(x_{i+2})$ (all indices taken modulo 3); see Figure~\ref{fig:BU-circle-to-tree-2}.
Consider the three shortest paths in $t$ between $f(x_i)$ and $f(x_{i+1})$, for all $i\in\{0,1,2\}$.
These three shortest paths form a tripod in $T$, i.e.\ a graph with one internal vertex and three ends.
Let $b\in T$ be the unique internal vertex (or branching point) of this tripod.
By the intermediate value theorem, for each $i\in\{0,1,2\}$ there exists some $y_i\in[x_i,x_{i+1}]\subseteq S^1$ with $f(y_i)=b$.
We claim $d(y_i,y_{i+1})\ge\frac{2\pi}{3}$ for some $i$.
If two of the these $y_i$ points are antipodal, then certainly the claim is true, and if not then we can consider the three geodesic arcs between the $y_i$ points.
If the union of the three geodesic arcs between the $y_i$ points do not cover $S^1$, as in Figure~\ref{fig:BU-circle-to-tree-2}(right), then one can take $y_i$ and $y_{i+1}$ to be the two endpoints of the geodesic convex hull of $\{y_0,y_1,y_2\}$.
Alternatively, if the union of the three geodesic arcs between the $y_i$ cover $S^1$, as in Figure~\ref{fig:BU-circle-to-tree-2}(left), then $d(y_0,y_1)+d(y_1,y_2)+d(y_2,y_0)=2\pi$ implies $d(y_i,y_{i+1})\ge\frac{2\pi}{3}$ for some $i$.
Hence there exists $i\in\{0,1,2\}$ with  $d(y_i,y_{i+1})\ge\frac{2\pi}{3}$ and $f(y_i)=b=f(y_{i+1})$.
\end{proof}

Note that the lower bound $\frac{2\pi}{3}$ in Lemma~\ref{lem:BU-circle-tree} cannot be improved, as shown in Figure~\ref{fig:circle_to_tripod}.

\begin{figure}[h!]
\centering
\begin{tikzpicture}
\draw[thick] (0,0) circle (2cm);
\draw[gray,->] (canvas polar cs:radius=1.8cm,angle=270)--(canvas polar cs:radius=0.2cm,angle=270); 
\draw[gray,->] (canvas polar cs:radius=1.8cm,angle=30)--(canvas polar cs:radius=0.2cm,angle=30)node[pos=0.5,above]{$f$};
\draw[gray,->] (canvas polar cs:radius=1.8cm,angle=150)--(canvas polar cs:radius=0.2cm,angle=150);
\fill[red] (canvas polar cs:radius=2cm,angle=30) circle (2pt) (canvas polar cs:radius=2cm,angle=270) circle (2pt) (canvas polar cs:radius=2cm,angle=150) circle (2pt);
\draw[red] (canvas polar cs:radius=2.3cm,angle=30) node{$a$};
\draw[red] (canvas polar cs:radius=2.3cm,angle=150) node{$b$};
\draw[red] (canvas polar cs:radius=2.3cm,angle=270) node{$c$};
\draw [<->,domain=35:145] plot ({2.3*cos(\x)}, {2.3*sin(\x)});
\draw [<->,domain=155:265] plot ({2.3*cos(\x)}, {2.3*sin(\x)});
\draw [<->,domain=275:385] plot ({2.3*cos(\x)}, {2.3*sin(\x)});
\draw [domain=120:180,gray,dashed] plot ({3.46*cos(\x)+1.73}, {3.46*sin(\x)-1});
\draw [domain=0:60,gray,dashed] plot ({3.46*cos(\x)-1.73}, {3.46*sin(\x)-1});
\draw [domain=240:300,gray,dashed] plot ({3.46*cos(\x)}, {3.46*sin(\x)+2});
\draw [domain=300:360,gray,dashed] plot ({3.46*cos(\x)-3.46}, {3.46*sin(\x)+2});
\draw [domain=60:120,gray,dashed] plot ({3.46*cos(\x)}, {3.46*sin(\x)-4});
\draw [domain=180:240,gray,dashed] plot ({3.46*cos(\x)+3.46}, {3.46*sin(\x)+2});
\draw[gray,dashed] (0,2)--(1.73,-1)--(-1.73,-1)--(0,2);
\draw (0,2.3) node[above]{$\tfrac{2\pi}{3}$};
\draw (canvas polar cs:radius=2.6cm,angle=210) node{$\tfrac{2\pi}{3}$};
\draw (canvas polar cs:radius=2.6cm,angle=330) node{$\tfrac{2\pi}{3}$};
\fill[blue] (0,0) circle (2pt);
\draw[blue,thick] (canvas cs:x=0,y=0)-- (canvas polar cs:radius=2cm,angle=90) (canvas cs:x=0,y=0)-- (canvas polar cs:radius=2cm,angle=210) (canvas cs:x=0,y=0)-- (canvas polar cs:radius=2cm,angle=330);
\end{tikzpicture}
\caption{The lower bound in Lemma~\ref{lem:BU-circle-tree} is tight.
The function $f$ sends the black circle to the blue tripod.
The points $a$, $b$ and $c$ are at distance $\frac{2\pi}{3}$ apart and in the same fiber of $f$.
}
\label{fig:circle_to_tripod}
\end{figure}

Recall that $G_0$ is the core of a metric graph $G$ (see Section~\ref{sec:background}), and that $e(G_0)$ is the length of the shortest edge in $G_0$.

\begin{lemma}
\label{lemma:geodesic_on_cycle}
Let $(G,d)$ be a metric graph and $\gamma$ be a simple loop.
If $y,y'\in\gamma$ satisfy $d(y,y')<\frac{e(G_0)}{2}$, then there is a unique geodesic joining $y$ and $y'$ and it lies on $\gamma$.
\end{lemma}

\begin{proof}
Let $\eta$ be a simple path connecting $y$ and $y'$.
If $\eta$ is not entirely in $\gamma$, its length is at least $e(G_0)$.
Hence, because $d(y,y')<\frac{e(G_0)}{2}$, a geodesic connecting the two points must lie on $\gamma$; see Figure~\ref{fig:lemma_no_shortcuts}.
Furthermore, this geodesic is unique since the length of a loop is at least $e(G_0)$.
\end{proof}

\begin{figure}[h!]
\centering
\begin{subfigure}{0.45\textwidth}
\centering
\begin{tikzpicture}
\draw (0,0) arc (90:-90:1.5);
\draw (-2,-3) arc (270:90:1.5);
\draw[red,thick,dashed] (0,0.1) arc (90:-90:1.6);
\draw[red,thick,dashed] (-2,-3.1) arc (270:90:1.6);
\draw (0,0)--(-2,0) (0,-3)--(-2,-3);
\draw[red,thick,dashed] (0,0.1)--(-2,0.1) node[pos=1,above]{$\gamma$} (0,-3.1)--(-2,-3.1);
\fill (-0.5,0) circle (2pt) (-0.5,-3) circle (2pt);
\draw (-0.5,0.1) node[above]{$y$};
\draw (-0.5,-3.1) node[below]{$y'$};
\draw (-1,0)--(-1,-3) node[pos=0.5,left]{$\eta$};
\end{tikzpicture}
\caption{}
\end{subfigure}
\begin{subfigure}{0.45\textwidth}
\centering
\begin{tikzpicture}
\draw (0,0) circle (2);
\draw (-2,0)--(2,0);
\draw (-0.7,0.05) node[above]{$y$};
\draw (0.7,0.05) node[above]{$y'$};
\draw[red, thick, dashed] (0,0) circle (2.1);
\draw[red, thick, dashed] (-2.1,-0.1)--(-0.7,-0.1)--(-0.7,0.1)--(-2.1,0.1) (2.1,-0.1)--(0.7,-0.1)--(0.7,0.1)--(2.1,0.1);
\draw[white,ultra thick] (-2.1,-0.1)--(-2.1,0.1) (2.1,-0.1)--(2.1,0.1);
\draw[red] (-1.5,1.5) node[above]{$\gamma$};
\draw (1,-1) node[below]{$G$};
\fill (-0.7,0) circle (2pt) (0.7,0) circle (2pt);
\end{tikzpicture}
\caption{}
\end{subfigure}
\caption{The assumptions in Lemma~\ref{lemma:geodesic_on_cycle} are necessary.
Loop $\gamma$ is represented with a dashed red line.
(A) If there is no bound on $e(G_0)$, then there may exist a shorter path connecting $y$ and $y'$ using a shortcut $\eta$.
(B) If $\gamma$ is not simple, then $y$ and $y'$ can be arbitrarily close and be joined by geodesics that are not in $\gamma$.}
\label{fig:lemma_no_shortcuts}
\end{figure}

The following lemma is the crucial step showing that a continuous map between finite graphs must send simple loops to simple loops, provided that the distortion is sufficiently small compared to the length of the shortest internal edge.

\begin{lemma}
\label{lemma:loop_injection}
\label{lemma:loops_permutation_continuous_map}
Let $\overline f\colon(G,d_G)\to (G',d_{G'})$ be a continuous map between finite metric graphs.
Assume that $\tfrac{\dis(\overline f)}{2}<\tfrac{\min\{e(G_0),e(G'_0)\}}{8}$.
Then $\overline f$ induces an injection $\Phi\colon\mathcal L(G)\to\mathcal L(G')$ between the simple loops of $G$ and the simple loops of $G'$ as follows: given a simple loop $\gamma\in\mathcal L(G)$, $\Phi(\gamma)$ is the only simple loop of $G'$ contained in $\overline f(\gamma)$.
\end{lemma}

\begin{proof}\renewcommand{\qedsymbol}{{\em Proof of Lemma~\ref{lemma:loops_permutation_continuous_map}} $\square$}
For readability purposes, let us divide the proof into several subclaims.
First, let us show that the map $\Phi$ is well-defined, i.e., for every simple loop $\gamma\in\mathcal L(G)$ there is a simple loop of $G'$ contained in $\overline f(\gamma)$ and that simple loop is unique.

\begin{subclaim}[Existence]\label{subclaim:existence}
If $\gamma$ is a simple loop of $G$, then $\overline f(\gamma)$ contains a simple loop of $G'$.
\end{subclaim}

\begin{proof}[Proof of Subclaim~\ref{subclaim:existence}]\renewcommand{\qedsymbol}{{\em Proof of Subclaim~\ref{subclaim:existence}} $\square$}
Suppose for a contradiction that $\overline{f}(\gamma)$ does not contain a loop.
Since $\overline{f}(\gamma)$ is connected, it is a tree.
So Lemma~\ref{lem:BU-circle-tree} applies to give points $x,x'\in \gamma$ with $\overline{f}(x)=\overline{f}(x')$ and $d_\gamma(x,x')\ge \tfrac{\mathrm{length}(\gamma)}{3}$ where $d_\gamma$ is the geodesic distance on $\gamma$.
We must have $d_G(x,x')\geq \tfrac{e(G_0)}{3}$, since if $d_G(x,x')<\tfrac{e(G_0)}{3}<\tfrac{e(G_0)}{2}$, then Lemma~\ref{lemma:geodesic_on_cycle} would give that a geodesic connecting $x$ and $x'$ lies on $\gamma$, and so $d_G(x,x')=d_\gamma(x,x')\ge \tfrac{\mathrm{length}(\gamma)}{3} \geq\tfrac{e(G_0)}{3}$.
We therefore obtain the contradiction
\[
\dis(\overline f)\geq |d_G(x,x')-d_{G'}(\overline{f}(x),\overline{f}(x'))|=d_G(x,x')\geq\tfrac{e(G_0)}{3}>\tfrac{e(G_0)}{4}>\dis(\overline f).
\]
Hence $\overline{f}(\gamma)$ must contain a loop.
\end{proof}

\begin{subclaim}[Uniqueness]\label{subclaim:uniqueness}
If $\gamma$ is a simple loop of $G$, then $\overline f(\gamma)$ contains at most one simple loop.
\end{subclaim}

\begin{proof}[Proof of Subclaim~\ref{subclaim:uniqueness}]\renewcommand{\qedsymbol}{{\em Proof of Subclaim~\ref{subclaim:uniqueness}} $\square$}
Note that $\overline f(\gamma)$ is a closed subset of $G'$ since it is the continuous image of a compact subset.
Let $H=\overline f(\gamma)$, which is a connected subgraph of $G'$.
Furthermore, let $H_0$ be the smallest connected subgraph of $G'$ containing the union of all simple loops of $H$.

Suppose for a contradiction that $\overline f(\gamma)=H$ contains at least two simple loops.
Therefore, at least one vertex of $H_0$ has degree at least $3$.
We can then consider the canonical graph representation of $H_0$ consisting of only vertices of degree at least $3$.
For every $v\in V(H_0)$, the closed subset $(\overline f)^{-1}(v)$ has diameter at most $\dis(\overline f)<\frac{e(G_0)}{2}$.
Hence, the convex hull $C_v$ of 
$(\overline f)^{-1}(v) \cap \gamma$ is well-defined.
Furthermore, for every $v,v'\in V(H_0)$ with $v\neq v'$, we have $C_v\cap C_{v'}=\emptyset$.
Indeed, for $x\in(\overline f)^{-1}(v)$ and $x'\in(\overline f)^{-1}(v')$ with $v\neq v'$, we have
\[
d_G(x,x')\geq d_{G'}(v,v')-\dis(\overline f)\geq e(G'_0)-\dis(\overline f)>\dis(\overline f)\geq\max\{\diam C_v,\diam C_{v'}\}.
\]
Let $\lvert V(H_0)\rvert=n$ and let $\gamma_1,\dots,\gamma_n$ be the non-empty connected components of $\gamma\setminus\bigcup_{v\in V(H_0)}C_v$.
For every $i\in\{1,\dots,n\}$, note that $\overline f(\gamma_i)$ is contained in a connected component of $H\setminus V(H_0)$.

\begin{figure}[h!]
    \centering
    \begin{tikzpicture}
        \draw (0,0) circle (2cm);
        \draw [red,ultra thick,domain=0:20] plot ({2*cos(\x)}, {2*sin(\x)});
        \draw[red] (canvas polar cs:radius=2.4cm,angle=10) node{$C_{v_1}$};
        \draw [red,ultra thick,domain=50:65] plot ({2*cos(\x)}, {2*sin(\x)});         \draw[red] (canvas polar cs:radius=2.4cm,angle=57.5) node{$C_{v_2}$};
       \draw [red,ultra thick,domain=100:120] plot ({2*cos(\x)}, {2*sin(\x)});        \draw[red] (canvas polar cs:radius=2.4cm,angle=110) node{$C_{v_3}$};        
       \draw [red,ultra thick,domain=160:174] plot ({2*cos(\x)}, {2*sin(\x)});        
       \draw [red,ultra thick,domain=204:226] plot ({2*cos(\x)}, {2*sin(\x)});        
       \draw [red,ultra thick,domain=260:280] plot ({2*cos(\x)}, {2*sin(\x)});            \draw [red,ultra thick,domain=310:320] plot ({2*cos(\x)}, {2*sin(\x)});
       \draw[red] (canvas polar cs:radius=2.4cm,angle=315) node{$C_{v_n}$};
       \draw[red] (canvas polar cs:radius=2.4cm,angle=215) node{$\cdots$};
       \draw (canvas polar cs:radius=1.6cm, angle=35) node{$\gamma_1$};
       \draw (canvas polar cs:radius=1.6cm, angle=82.5) node{$\gamma_2$};
       \draw (canvas polar cs:radius=1.6cm, angle=140) node{$\gamma_3$};
       \draw (canvas polar cs:radius=1.6cm, angle=340) node{$\gamma_n$};
       \draw (canvas polar cs:radius=1.6cm, angle=240) node{$\cdots$};
       \draw (-2,2) node{$\gamma$};

       \draw[ultra thick, blue] (8,0) circle (1.5cm);
       \draw[ultra thick,blue] (9.5,0.5) circle (1cm);
       \draw (7,-2)--(7,2) (6,2)--(10,-1.75) (8,-1.5)--(10.5,-1.5);
       \draw[ultra thick,blue] (7,-1.1)--(7,1.1) (8,-1.5)--(9.75,-1.5) (9.75,-1.5)--(7,1.1) (7,-1.1)--(8,0.17);
       \fill[red] (7,1.1) circle (2pt) (8.85,1.24) circle (2pt) (8,-1.5) circle (2pt) (7,-1.1) circle (2pt) (9.4,-0.5) circle (2pt) (9.15,-0.95) circle (2pt) (7.99,0.165) circle (2pt);
       \draw[red] (9.4,-0.5) node[below right]{$v_3$};
       \draw[red] (7,1.1) node[right]{$v_1$};
       \draw[red](8.85,1.24) node[above]{$v_2$};
       \draw[red] (7,-1.1) node[below left]{$v_n$};
       \draw[blue] (11,0) node{$H_0$};
       \draw (6,1.5) node[below]{$H=\overline f(\gamma)$};
       \draw[blue] (8,1.75) node{$\eta_1$};
       \draw[blue] (8.8,0.25) node{$\eta_2$};
       \draw[blue] (9.7,0.5) node{$\eta_3$};
       \draw[blue] (6.2,0) node{$\eta_m$};
        \draw[->] (3,0)--(5,0)node[pos=0.5,above]{$\overline f$};
    \end{tikzpicture}
    \caption{A representation of the objects constructed in the proof of Subclaim~\ref{subclaim:uniqueness}.}
    \label{fig:uniqueness}
\end{figure}

Let $\eta_1,\dots,\eta_m$ be the connected components of $H_0\setminus V(H_0)$.
Note that each $\eta_j$ is an open-ended path in $H_0$.
We claim that $\eta_j\cap\overline f(\gamma_1\cup\cdots\cup\gamma_n)\neq\emptyset$ for every $j\in\{1,\dots,m\}$.
Let $y$ be the midpoint of $\eta_j$.
Since $H=\overline f(\gamma)$, there is $z\in\gamma$ such that $\overline f(z)=y$.
For every $v\in V(H_0)$ and every $x\in (\overline f)^{-1}(v)$, 
\[d_G(z,x)\geq d_{G'}(y,v)-\dis(\overline f)\geq\tfrac{e(G'_0)}{2}-\dis(\overline f)>2\dis(\overline f)-\dis(\overline f)=\dis(\overline f)\geq\diam C_v.\]
Hence, $z\notin C_v$ and so $z\in\gamma_1\cup\cdots\cup\gamma_n$.

So far, we have shown that each $\overline f(\gamma_i)$, for $i\in\{1,\dots,n\}$, is contained in some connected component of $H\setminus V(H_0)$.
Furthermore, among these connected components, each $\eta_j$ with $j\in\{1,\dots,m\}$ contains the $\overline f$-image of at least one $\gamma_i$.
It will be a contradiction if we prove that $n<m$.
Note that the set $\{\eta_j\}_{j=1}^m$ is in bijection with the edge set $E(H_0)$ of $H_0$.
Hence 
\[2m=2\lvert E(G'_0)\rvert=\sum_{v\in V(G'_0)}\deg_{G'_0}(v)\geq 3\lvert V(G'_0)\rvert=3n,\]
giving the contradiction $n<m$.
So it must be that $\overline f(\gamma)$ contains at most one simple loop.
\end{proof}

According to Subclaims~\ref{subclaim:existence} and~\ref{subclaim:uniqueness}, the map $\Phi$ is well-defined.
Lastly, we show that it is injective.

\begin{subclaim}
\label{subclaim:injective}
The map $\Phi\colon\mathcal L(G)\to\mathcal L(G')$ is injective.
\end{subclaim}

\begin{proof}[Proof of Subclaim~\ref{subclaim:injective}]\renewcommand{\qedsymbol}{{\em Proof of Subclaim~\ref{subclaim:injective}} $\square$}
Let $\gamma_1$, $\gamma_2$, and $\eta$ be three simple loops such that $\eta\subseteq\overline f(\gamma_1)\cap\overline f(\gamma_2)$.
Suppose, by contradiction, that $\gamma_1$ and $\gamma_2$ are distinct.
Since they are simple, they differ by at least one edge.
Without loss of generality, let us fix an edge $e\in E(G)$ that is contained in $\gamma_1$ but whose interior is disjoint from $\gamma_2$.
Fix the midpoint $m$ of $e$.
In particular, $d_G(m,\gamma_2)\geq \frac{e(G_0)}{2}$.

\begin{figure}[h!]
    \centering
    \begin{tikzpicture}
        \draw (0,0) circle (1.5cm) (3,0) circle (1.5cm);
        \draw[red,ultra thick] (9.5,0) circle (1.5cm);
        \draw[red] (9.5,1.5) node[above]{$\eta$};
        \draw[->] (5,0)--(6.5,0)node[pos=0.5,above]{$\overline f$};
        \draw[ultra thick,red,domain=300:420] plot ({1.5*cos(\x)}, {1.5*sin(\x)});
        \draw[ultra thick,red,domain=425:430] plot ({1.5*cos(\x)}, {1.5*sin(\x)});
        \draw[ultra thick,red,domain=435:440] plot ({1.5*cos(\x)}, {1.5*sin(\x)});
        \draw[ultra thick,red,domain=280:290] plot ({1.5*cos(\x)}, {1.5*sin(\x)});
        \draw[ultra thick,red,domain=130:220] plot ({1.5*cos(\x)+3}, {1.5*sin(\x)});
        \draw[line width=8,blue,opacity=0.2,domain=80:280] plot ({1.5*cos(\x)}, {1.5*sin(\x)});
        \draw[blue] (canvas polar cs:radius=1.8cm,angle=120) node[left]{$C$};
        \draw[red] (1.4,0) node[left]{$\overline f^{-1}(\eta)$};
        \fill (-1.5,0) circle (2pt);
        \draw (-1.5,0) node[left]{$m$};
        \fill (canvas polar cs:radius=1.5cm,angle=80) circle (2pt)  (canvas polar cs:radius=1.5cm,angle=280) circle (2pt);
        \draw (canvas polar cs:radius=1.5cm,angle=80) node[above]{$x_1$};
        \draw (canvas polar cs:radius=1.5cm,angle=280) node[below]{$x_2$};
        \draw (canvas polar cs:radius=1.7cm,angle=240) node[below]{$\gamma_1$};
        \draw (3,0)+(canvas polar cs:radius=1.7cm,angle=300) node[below]{$\gamma_2$};
        \draw (9.5,0)+(canvas polar cs:radius=1.5cm,angle=120)--(8,1.75)--(7.3,1.5);
        \draw (8,1.75)--(7.2,2.3)node[pos=0.5,above right]{$\overline f(\gamma_2)$};
        \draw (9.5,0)+(canvas polar cs:radius=1.5cm,angle=230)--(7.4,-2);
        \draw (9.5,0)+(canvas polar cs:radius=1.5cm,angle=290)--(10.5,-2.5);
        \draw (9.5,0)+(canvas polar cs:radius=1.5cm,angle=30)--(11.5,1)--(12.5,0.5) (11.5,1)--(11.7,2)--(12.5,2.1) (11.7,2)--(12,2.5);
        \fill (9.5,0)+(canvas polar cs:radius=1.5cm,angle=30) circle (2pt);
        \draw (9.5,0)+(canvas polar cs:radius=1.5cm,angle=30) node[left]{$z=\overline f(x_1)=\overline f(x_2)$};
        \draw
        [line width=8,blue,opacity=0.2] (9.5,0)+(canvas polar cs:radius=1.5cm,angle=30)--(11.5,1)--(12.5,0.5) (11.5,1)--(11.7,2)--(12.5,2.1) (11.7,2)--(12,2.5);
        \draw[blue] (12.5,0.5) node[below]{$D\supseteq \overline f(C)$};
    \end{tikzpicture}
\caption{A representation of the objects constructed in the proof of Subclaim~\ref{subclaim:injective}.}
\label{fig:injectivity}
\end{figure}

Let us estimate $d_G(m,(\overline f)^{-1}(\eta))$.
Suppose that $x\in(\overline f)^{-1}(\eta)$, and fix $y\in\gamma_2$ such that $\overline f(x)=\overline f(y)$.
Note that $d_G(x,y)\leq\dis(\overline f)$.
Using the triangle inequality, we obtain
\begin{equation}
\label{eq:dist_f-1eta}
\begin{aligned}
d_G(m,x)\,&\geq d_G(m,y)-d_G(x,y)\geq d_G(m,\gamma_2)-\dis(\overline f)\\
&\geq\tfrac{e(G_0)}{2}-\dis(\overline f)>2\dis(\overline f)-\dis(\overline f)=\dis(\overline f).
\end{aligned}
\end{equation}
Hence $d(m,(\overline f)^{-1}(\eta))\geq\dis(\overline f)$, and in particular, $\overline f(m)\notin\eta$.

Consider the connected component $C$ of $\gamma_1\setminus(\overline f)^{-1}(\eta)$ containing $m$.
Since $(\overline f)^{-1}(\eta)$ is closed by continuity, $C$ is an open arc in $\gamma_1$.
Let $x_1$ and $x_2$ be the two points of $\gamma_1$ on the boundary of $C$.
Note that they both belong to $(\overline f)^{-1}(\eta)$, and so $d_G(m,x_i)
>\dis(\overline f)$ by~\eqref{eq:dist_f-1eta}.

Let us estimate $d_G(x_1,x_2)$.
A geodesic connecting $x_1$ and $x_2$ satisfies one of the following:
\begin{compactenum}[(i)]
    \item either it passes through $m$,
    \item or it crosses $(\overline f)^{-1}(\eta)\cap\gamma_1$,
    \item or it uses an edge that is not in $\gamma_1$.
\end{compactenum}
In all three cases, we have $d_G(x_1,x_2)>\dis(\overline f)$.
For (i), this follows since $d_G(m,x_i)>\dis(\overline f)$.
For (ii), let $A\coloneqq(\overline f)^{-1}(\eta)\cap\gamma_1$.
Since $\overline f(\gamma_1)$ contains $\eta$, we have $\overline f(A)=\eta$.
The diameter of $\eta$ is at least $\tfrac{e(G'_0)}{2}$, and hence
\[\diam(A) \ge \tfrac{e(G'_0)}{2}-\dis(\overline f)  > 2\dis(\overline f)-\dis(\overline f) = \dis(\overline f).\]
Finally, we also have $d_G(x_1,x_2)>\dis(\overline f)$ in case (iii) since $e(G_0) > \dis(\overline f)$.

Since $\overline f$ is continuous, $\overline f(C)$ is contained in a connected component $D$ of $\overline f(\gamma_1)\setminus\eta$.
Since $\overline f(\gamma_1)$ contains only one loop, $D$ is a tree.
Furthermore, $\overline f(\gamma_1)$ is a closed subset of $G'$ since it is the continuous image of a compact subset, and it contains only one simple loop.
Thus, the closure of $D$ in $G'$ intersects $\eta$ in only one point $z$.
We have already observed that $\overline f(x_i)\in\eta$ for $i\in\{1,2\}$.
Furthermore, $\overline f(x_i)$ belongs to the closure of $\overline f(C)$ in $G'$, which is contained in the closure of $D$.
Therefore, $\overline f(x_i)=z$ since the latter is the only point in the closure of $D$ that lies on $\eta$.
This gives $\overline f(x_1)=z=\overline f(x_2)$, which contradicts the above established inequality $d_G(x_1,x_2)>\dis(\overline f)$.
\end{proof}

This completes the proof of Lemma~\ref{lemma:loops_permutation_continuous_map}.
\end{proof}

\subsection{Distortion of arbitrary functions between graphs}

To apply Lemma~\ref{lemma:loops_permutation_continuous_map} in the context of Gromov--Hausdorff distances, we will need a way to turn arbitrary functions into nearby continuous ones with controlled distortion, as described in the following result.

\begin{lemma}
\label{lemma:from_discontinuous_to_continuous}
Let $f\colon (G,d_G)\to (G',d_{G'})$ be a function between finite metric graphs with distortion $\tfrac{\dis(f)}{2}<r<\tfrac{\min\{e(G_0),e(G'_0)\}}{4}$.
Then there is a continuous map $\overline f\colon G\to G'$ such that $\tfrac{\dis(\overline f)}{2}<3r$ and $\overline f(G)\subseteq (f(G))^{r}$.
\end{lemma}

\begin{proof}
We organize the proof into three steps, to help its readability.
Let $\varepsilon>0$ satisfy $\varepsilon+\dis(f)<2r$.

\smallskip

\noindent {\bf Step 1: Construct a continuous map $\overline f$.}
Fix a finite triangulation $K$ of $G$ with edges of length at most $\varepsilon$ such that $V(G)\subseteq K^{(0)}$.
Let us construct a continuous map $\overline f\colon K\to G'$ as follows.
For every vertex $x\in K^{(0)}$, set $\overline f(x)=f(x)$.
Consider now two adjacent vertices $x$ and $y$ in $K^{(0)}$.
Send the segment from $x$ to $y$ to the geodesic segment connecting $f(x)$ and $f(y)$ at uniform speed.
The geodesic between the two points $f(x)$ and $f(y)$ is unique since $d_{G'}(f(x),f(y))\leq\varepsilon+\dis(f)<\tfrac{e(G'_0)}{2}$.

\smallskip

\noindent {\bf Step 2: Estimate the distortion of $\overline f$.}
We claim that $\dis(\overline f)\leq 3(\dis(f)+\varepsilon)<6r$.
The second inequality is trivial.
Let $x,y\in K$.
By construction of $K$, we can let $x_i,y_j\in K^{(0)}$ be such that $d_G(x,x_i)\leq\frac{\varepsilon}{2}$ and $d_G(y,y_j)\leq\frac{\varepsilon}{2}$.
By construction of $\overline f$, we have $d_{G'}(\overline f(x),\overline f(x_i))\leq \dis(f)+\varepsilon$ since $\overline f(x)$ is contained in the geodesic between $f(x_i)$ and $f(x_{i+1})$, where $x$ is contained in the edge $\{x_i,x_{i+1}\}$ of $K$.
Similarly, $d_{G'}(\overline f(y),\overline f(y_j))\leq\dis(f)+\varepsilon$.
By applying the triangle inequality, we obtain the following chain of inequalities:
\begin{align*}
    d_{G'}(\overline f(x),\overline f(y))\,&\leq d_{G'}(\overline f(x),\overline f(x_i))+d_{G'}(f(x_i),f(y_j))+d_{G'}(\overline f(y_j),\overline f(y))\\
    &\leq \dis(f)+\varepsilon+ d_G(x_i,y_j)+\dis(f)+\dis(f)+\varepsilon\\
    &\leq 3\dis(f)+2\varepsilon +d_G(x_i,x)+d_G(x,y)+d_G(y,y_j)\\
    &\leq 3(\dis(f)+\varepsilon)+d_G(x,y)\\
    &<6r+d_G(x,y).
\end{align*}
With similar computations, we can show that $d_G(x,y)-d_{G'}(\overline f(x),\overline f(y))\leq 3(\dis(f)+\varepsilon)<6r$.
This proves $\tfrac{\dis(\overline f)}{2}<3r$.

\smallskip

\noindent {\bf Step 3: Show $\overline f(K)\subseteq(f(G))^r$.}
For every point $\overline f(x)\in \overline f(K)$, choose a closest point $x_i \in K^{(0)}$ to $x$.
Let $x_{i+1}\in K^{(0)}$ be such that $x$ is contained in the edge $\{x_i,x_{i+1}\}$.
Thus, by construction of $\overline f$, 
\[
d_{G'}(\overline f(x),f(x_i))\leq\frac{d_{G'}(f(x_i),f(x_{i+1}))}{2}\leq \frac{\varepsilon+\dis(f)}{2}<r\qedhere
\]
\end{proof}

We finally obtain that the core $G_0$ of a graph $G$ is preserved by a function with distortion bounded by a fraction of the shortest edge length.

\begin{proposition}
\label{prop:G_0_in_f(G_0)}
Let $G$ be a finite metric graph and let $f\colon G\to G$ be a function with $\tfrac{\dis(f)}{2}<r<\tfrac{e(G_0)}{24}$.
Then $G_0\subseteq f(G_0)^{r}$.
\end{proposition}

\begin{proof}
Let $\overline f\colon G\to G$ be the continuous map described in Lemma~\ref{lemma:from_discontinuous_to_continuous}.
Then $\tfrac{\dis(\overline f)}{2}<3r<\tfrac{e(G_0)}{8}$.
The restriction $\overline f|_{G_0}\colon G_0\to G$ is also continuous and has the same bound on its distortion.
We can then apply Lemma~\ref{lemma:loop_injection} to $\overline f|_{G_0}$ and obtain an injection $\Phi\colon\mathcal L(G_0)\to\mathcal L(G)$.
Since $\mathcal L(G_0)=\mathcal L(G)$ by construction and these sets are finite, $\Phi$ is a bijection.
Hence, every simple loop of $G$ is contained in $\overline f(G_0)\subseteq f(G_0)^r$ according to Lemma~\ref{lemma:from_discontinuous_to_continuous}.
Given that $f(G_0)^{r}$ is connected, it contains $G_0$.
\end{proof}

As an immediate consequence of Proposition~\ref{prop:G_0_in_f(G_0)} we obtain that the Hausdorff and the Gromov--Hausdorff distances between a graph with no leaves and a subset thereof coincide, provided that the latter is sufficiently small.

\begin{corollary}
\label{cor:no-leaves}
Let $G$ be a finite metric graph with no leaves, and let $X\subseteq G$.
If $d_\gh(G,X)<\tfrac{e(G_0)}{24}$, then $d_\gh(G,X)=d_\h(G,X)$.
\end{corollary}

\begin{proof}
Let $r>0$ satisfy $d_\gh(G,X)<r<\tfrac{e(G_0)}{24}$.
Then there is a map $f\colon G\to X$ whose distortion satisfies $\tfrac{\dis(f)}{2}<r<\tfrac{e(G_0)}{24}$.
We can apply Proposition~\ref{prop:G_0_in_f(G_0)} to obtain that $G=G_0\subseteq f(G_0)^r\subseteq X^r$.
Hence $d_\h(G,X)\leq r$.
Since $r$ can be taken arbitrarily, we have $d_\h(G,X)\leq d_\gh(G,X)$.
\end{proof}

In the rest of the paper, we discuss how to lower bound the Gromov--Hausdorff distance with respect to the Hausdorff distance for graphs with leaves.
The first result relies on hypotheses around the subset $X$ of the graph $G$.
More precisely, similarly to the tree case described in Proposition~\ref{thm:tree_via_maps} and Theorem~\ref{thm:maps_and_trees}, we request that the subset $X$ is sufficiently close to the leaves of $G$.

\begin{proposition}
\label{prop:HvsGH_graphs}
Let $(G,d)$ be a finite metric graph
and let $f\colon G\to G$ be a (possibly discontinuous) function with $\tfrac{\dis(f)}{2}<\tfrac{e(G_0)}{24}$.
If $\partial G\neq \emptyset$, suppose that $d_\h(G,f(G))>\vec{d}_\h(\partial G,f(G))$.
Then $d_\h(G,f(G))\leq \tfrac{\dis(f)}{2}$.
\end{proposition}



\begin{proof}
Let $r>0$ satisfy $\tfrac{\dis(f)}{2}<r<\tfrac{e(G_0)}{24}$.
Let $\overline x\in G$ be a point such that $B(\overline x, d_\h(G,f(G)))\cap f(G)=\emptyset$.
If $\overline x\in G_0$, then $\overline x\in f(G)^r$ according to Proposition~\ref{prop:G_0_in_f(G_0)}, and so $d_\h(G,f(G))<r$.
Suppose otherwise that $\overline x\notin G_0$.
Thus, $G\setminus\{\overline x\}$ is disconnected, and at least two connected components thereof have non-empty intersection with $X$ by the assumption.
We can conclude as in Proposition~\ref{thm:tree_via_maps}.
\end{proof}

\begin{theorem}
\label{thm:graphs}
Let $G$ be a finite connected metric graph and let $X\subseteq G$.
If $\partial G\neq \emptyset$, suppose $d_\h(G,X)>\vec{d}_\h(\partial G,X)$.
If $d_\gh(G,X)<\tfrac{e(G_0)}{24}$, then $d_\gh(G,X)= d_\h(G,X)$.
\end{theorem}

\begin{proof}
As in the proof of Theorem~\ref{thm:maps_and_trees}, by Lemma~\ref{lemma:H_GH_and_compactness} we can assume that $X$ is compact by replacing it with its closure $\overline X$ if necessary.
Suppose, by contradiction, that $d_\gh(G,X)<d_\h(G,X)$.
By Lemma~\ref{lemma:map_close_to_the_leaves} there is a function $f\colon G\to X$ such that $\tfrac{1}{2}\dis(f)<\min\{d_\h(G,X),\tfrac{e(G_0)}{24}\}$ and $\vec d_\h(\partial G,X)=\vec d_\h(\partial G,f(G))$.
Note that $d_\h(G,f(G))\geq d_\h(G,X)>\vec d_\h(\partial G,X)=\vec d_\h(\partial G,f(G))$.
We can apply Proposition~\ref{prop:HvsGH_graphs} to obtain a contradiction with the following chain of inequalities:
\[
d_\h(G,X)\leq d_\h(G,f(G))\leq \tfrac{\dis(f)}{2}\leq d_{\gh}(G,X).
\qedhere
\]
\end{proof}


\section{Fractional bounds with no assumption on the subset of the graph}
\label{sec:no-boundary-assumptions}

The results in Sections~\ref{sec:trees}--\ref{sec:graphs} provide conditions under which $d_\gh(G,X)=d_\h(G,X)$ for a subset $X$ of a metric graph $G$.
They rely on the assumption that $d_\h(G,X)>\vec{d}_\h(\partial G,X)$.
This assumption, although generic for uniform samples (Proposition~\ref{prop:assumption_on_leaves_generic}), is never satisfied if $X$ is a connected proper subgraph of a metric tree $G$, or a sufficiently dense sample thereof.
A motivating example could be $T$ and $X$ from Figure~\ref{Fig:1}, to which none of our previous results apply.
An even simpler example would be the interval $X=[0,t]$ within $G=[0,1]$ for $t\in [0,1)$.
It is clear that $d_\h(G,X) = 1-t = 2 d_\gh(G,X)$; see Corollary~\ref{coro:segment}.
So only a fractional result $d_\gh(G,X)\ge\frac{1}{2}d_\h(G,X)$ is attainable in this case.

In this section, we prove fractional bounds: we show that $d_\gh(G,X)$ is greater than or equal to a fraction of $d_\h(G,X)$ under more general hypotheses.
In particular, we pose no assumptions on $X$, at the cost of weaker fractional bounds on $d_\gh(G,X)$.

We divide the section into two parts: in Subsection~\ref{sub:trees_no_assumptions}, we prove two such results for trees, and then we extend those techniques to adapt them to arbitrary graphs in Subsection~\ref{sub:graphs_no_assumptions}.

We will make extensive use of the following simple result.

\begin{lemma}
\label{lemma:ziga}
Let $Y$ be a subset of a metric space $X$.
Let $f\colon X\to X$ be a function and let $\varepsilon> 0$.
If $x\in X$ satisfies $f(x)\in f(Y)^\varepsilon$, then $d(x,Y)<\varepsilon+\dis(f)$.
\end{lemma}

\begin{proof}
Indeed, if there is a point $y\in Y$ satisfying $d(f(x),f(y))<\varepsilon$, then $d(x,y)<\varepsilon+\dis(f)$.
\end{proof}

For a metric space $X$, we denote by $\mathcal D_X$ its {\em distance set}, which is the set of all the nonzero distances $d(x,x')$ where $x,x'\in X$.
Furthermore, for every $x\in X$, we let $\mathcal D_X(x)$ denote the {\em local distance set at $x$}, namely $\mathcal D_X(x)=\{d(x,x')\mid x'\in X\}$.

For a finite subset $Y\subseteq \R$, let $\sep Y=\min_{y,y'\in Y,\,y\neq y'} |y-y'|$ be the {\em separation of $Y$}.
When $\mathcal D_X$ is a finite set, we will sometimes consider $\sep \mathcal D_X$.

For a finite graph $G$ with leaves, we let $\lambda(G)$ be the length of the shortest leaf edge.

\subsection{Trees with no assumptions on the subset}
\label{sub:trees_no_assumptions}

\begin{proposition}
\label{prop:trees_maximal_geodesics}
For $T$ a finite tree, let function $f\colon T\to T$ satisfy $\tfrac{\dis(f)}{2}<r<\min\bigl\{\tfrac{\lambda(T)}{4},\tfrac{\sep\mathcal D_{\partial T}}{2}\bigr\}$.
Then $T\subseteq f(T)^{2r}$.
\end{proposition}

\begin{proof}
Let us arrange the elements of $\mathcal D_{\partial T}$ in decreasing order: $c_1>c_2>\cdots>c_n$.
For every $i\in\{1,\dots,n\}$, define $\partial_iT=\{v\in\partial T\mid\max\mathcal D_{\partial T}(v)=c_i\}$.
Clearly, $\{\partial_iT\}_i$ forms a partition of $\partial T$.

Let us denote by $T_1$ the convex hull of $\partial_1T$, and by $T_{i+1}$ the union of geodesics connecting $\partial_{i+1}T$ to $\bigcup_{j\leq i}T_j$; see Figure~\ref{fig:trees_partition}.
Note that $\partial_1T$ contains at least two leaves, so its convex hull is not a singleton.

We intend to iteratively construct bijections $\Phi_i\colon\partial_iT\to\partial_iT$ for $i\in\{1,\dots,n\}$ such that for every $v\in\partial_iT$, we have $d(f(v),\Phi_i(v))<2r$.
Once this is accomplished, then for every $i\in\{1,\dots,n\}$, we have $\bigcup_{j\leq i}T_j\subseteq f(\bigcup_{j\leq i}T_j)^{2r}$ since the latter is path connected (Lemma~\ref{lemma:path_connected_distortion2}) and contains all its leaves, i.e., the set $\bigcup_{j\leq i}\partial_jT$.
Hence, the claim follows since $T=\bigcup_{j\leq n}T_j$.

\begin{figure}[h!]
\centering
\begin{subfigure}{0.23\textwidth}
\centering
\begin{tikzpicture}[scale=0.72]
\draw (-2,0)--(3,0) (-1,-1.5)--(-1,1) (-1,-1)--(-1.5,-1) (0,1)--(0,-1) (1.5,-1.5)--(1.5,1.5) (1,1)--(2.5,1);
\draw[ultra thick,red,dashed] (-1.5,-1)--(-1,-1) (-1,-1.5)--(-1,0)--(1.5,0)--(1.5,1)--(2.5,1);
\draw[red] (0.75,0) node[below]{$T_1$};
\draw[red] (-1,-1.5) circle (3.5pt) (-1.5,-1) circle (3.5pt) (2.5,1) circle (3.5pt);
\draw (-2,0) node{\tiny$\square$};
\draw (-1,1) node{\tiny$\square$};
\draw (0,1) node{\tiny$\triangle$};
\draw (0,-1) node{\tiny$\triangle$};
\draw (1.5,1.5) node{\small$\times$};
\draw (1,1) node{\small$\times$};
\draw (3,0) node{\small$\times$};
\draw (1.5,-1.5) node{\small$\times$};
\draw (-1.9,0.9) node{$T$};
\end{tikzpicture}
\end{subfigure}\quad
\begin{subfigure}{0.23\textwidth}
\centering
\begin{tikzpicture}[scale=0.72]
\draw (-2,0)--(3,0) (-1,-1.5)--(-1,1) (-1,-1)--(-1.5,-1) (0,1)--(0,-1) (1.5,-1.5)--(1.5,1.5) (1,1)--(2.5,1);
\draw[ultra thick,red] (-1.5,-1)--(-1,-1) (-1,-1.5)--(-1,0)--(1.5,0)--(1.5,1)--(2.5,1);
\draw[ultra thick,red,dashed] (1,1)--(1.5,1)--(1.5,1.5) (3,0)--(1.5,0)--(1.5,-1.5);
\draw[red] (-1,-1.5) circle (3.5pt) (-1.5,-1) circle (3.5pt) (2.5,1) circle (3.5pt);
\draw (-2,0) node{\tiny$\square$};
\draw (-1,1) node{\tiny$\square$};
\draw (0,1) node{\tiny$\triangle$};
\draw (0,-1) node{\tiny$\triangle$};
\draw[red] (1.5,1.5) node{\small$\times$};
\draw[red] (1,1) node{\small$\times$};
\draw[red] (3,0) node{\small$\times$};
\draw[red] (1.5,-1.5) node{\small$\times$};
\draw[red] (1.5,-1) node[right]{$T_2$};
\end{tikzpicture}
\end{subfigure}
\quad
\begin{subfigure}{0.23\textwidth}
\centering
\begin{tikzpicture}[scale=0.72]
\draw (-2,0)--(3,0) (-1,-1.5)--(-1,1) (-1,-1)--(-1.5,-1) (0,1)--(0,-1) (1.5,-1.5)--(1.5,1.5) (1,1)--(2.5,1);
\draw[ultra thick,red] (-1.5,-1)--(-1,-1) (-1,-1.5)--(-1,0)--(1.5,0)--(1.5,1)--(2.5,1) (1,1)--(1.5,1)--(1.5,1.5) (3,0)--(1.5,0)--(1.5,-1.5);
\draw[ultra thick,red,dashed] (-2,0)--(-1,0)--(-1,1);
\draw[red] (-1,-1.5) circle (3.5pt) (-1.5,-1) circle (3.5pt) (2.5,1) circle (3.5pt);
\draw[red] (-2,0) node{\tiny$\square$};
\draw[red] (-1,1) node{\tiny$\square$};
\draw (0,1) node{\tiny$\triangle$};
\draw (0,-1) node{\tiny$\triangle$};
\draw[red] (1.5,1.5) node{\small$\times$};
\draw[red] (1,1) node{\small$\times$};
\draw[red] (3,0) node{\small$\times$};
\draw[red] (1.5,-1.5) node{\small$\times$};
\draw[red] (-1,0.75) node[left]{$T_3$};
\end{tikzpicture}
\end{subfigure}\quad
\begin{subfigure}{0.23\textwidth}
\centering
\begin{tikzpicture}[scale=0.72]
\draw (-2,0)--(3,0) (-1,-1.5)--(-1,1) (-1,-1)--(-1.5,-1) (0,1)--(0,-1) (1.5,-1.5)--(1.5,1.5) (1,1)--(2.5,1);
\draw[ultra thick,red] (-1.5,-1)--(-1,-1) (-1,-1.5)--(-1,0)--(1.5,0)--(1.5,1)--(2.5,1) (1,1)--(1.5,1)--(1.5,1.5) (3,0)--(1.5,0)--(1.5,-1.5) (-2,0)--(-1,0)--(-1,1);
\draw[ultra thick,red,dashed] (0,1)--(0,-1);
\draw[red] (-1,-1.5) circle (3.5pt) (-1.5,-1) circle (3.5pt) (2.5,1) circle (3.5pt);
\draw[red] (-2,0) node{\tiny$\square$};
\draw[red] (-1,1) node{\tiny$\square$};
\draw[red] (0,1) node{\tiny$\triangle$};
\draw[red] (0,-1) node{\tiny$\triangle$};
\draw[red] (1.5,1.5) node{\small$\times$};
\draw[red] (1,1) node{\small$\times$};
\draw[red] (3,0) node{\small$\times$};
\draw[red] (1.5,-1.5) node{\small$\times$};
\draw[red] (0,-0.75) node[right]{$T_4$};
\end{tikzpicture}
\end{subfigure}
\caption{A finite tree $T$ and its partition $T_1,T_2,T_3,T_4$ as in the proof of Proposition~\ref{prop:trees_maximal_geodesics}.
The leaves in $\partial_1T$, $\partial_2T$, $\partial_3T$ and $\partial_4T$ are represented with circles, crosses, squares, and triangles, respectively.
The dashed lines are $T_1$, $T_2$, $T_3$ and $T_4$.}
\label{fig:trees_partition}
\end{figure}

Before describing the iterative construction, we formalize the following important observation: given $x\in T_i$ and $y\in T_j$,
\begin{equation}
\label{eq:diam}
d(x,y)\leq\min\{\max\mathcal D_{\partial_iT},\max\mathcal D_{\partial_jT}\}=\min\{c_i,c_j\}=c_{\max\{i,j\}}.
\end{equation}

\noindent{\bf Base step.} Let $x\in\partial_1T$.
Pick $y\in\partial_1T$ such that $d(x,y)=c_1$.
Since $\sep\mathcal D_{\partial T}>2r$,
\[
d(f(x),f(y))\geq c_1-\dis(f)>c_1-2r>c_2.
\]
Therefore, $f(x)$ (and similarly $f(y)$) belongs to a geodesic connecting two leaves in $\partial_1T$.
Indeed, if it were not the case, \eqref{eq:diam} would imply that $d(f(x),f(y))\leq c_2$.
Since $d(f(x),f(y))>c_1-2r$, there is an $x'\in\partial_1T$ with $d(f(x),x')\leq\dis(f)<2r$.
Such an $x'$ is the unique leaf with that property since the distance between every two leaves is strictly greater than $4r$.
Note that $\min\mathcal D_{\partial T}\geq2\lambda(T)$.
Define $\Phi_1(x)=x'$.
We claim that $\Phi_1$ is injective.
Indeed, if $\Phi_1(x)=\Phi_1(y)$ for $x,y\in\partial_1T$, then
\[
d(x,y)\leq d(f(x),f(y))+\dis(f)\leq d(f(x),\Phi_1(x))+d(\Phi_1(y),f(y))+\dis(f)<6r<2\lambda(T)\leq\min\mathcal D_{\partial T},
\]
and so $x=y$.
Since $\partial_1T$ is finite, $\Phi_1$ is bijective.

\noindent{\bf Recursive step.} Assume that bijections $\Phi_1,\dots,\Phi_i$ with the desired properties have been constructed.
Let $x\in\partial_{i+1}T$.
Note that $d(x,\bigcup_{j\leq i}T_j)\geq\lambda(T)>4r$ and so $f(x)\notin \bigcup_{j\leq i}T_j\subseteq f(\bigcup_{j\leq i}T_j)^{2r}$ according to Lemma~\ref{lemma:ziga} (take $\varepsilon=2r$).
Let $y\in\partial T$ be such that $d(x,y)=c_{i+1}$.
If $f(x)$ belongs to some $T_j$ with $j>i+1$, then \eqref{eq:diam} implies the contradiction
\[
c_{i+1}-2r< d(x,y)-\dis(f)\leq d(f(x),f(y))\leq c_j\leq c_{i+1}-\sep\mathcal D_{\partial T}<c_{i+1}-2r.
\]
Therefore, $f(x)\in T_{i+1}$.
Since $d(f(x),f(y))> c_{i+1}-2r$, there is some $x'\in\partial_iT$ with $d(f(x),x')<2r$.
Set $\Phi_{i+1}(x)=x'$.
Similarly to the base step, we can show that $\Phi_{i+1}$ is indeed well-defined and injective, thus concluding the proof.
\end{proof}

\begin{theorem}
\label{thm:trees_H_vs_GH_maximal_geodesics}
Let $T$ be a finite tree and let $X$ be a subset thereof.
Then
\[
d_\gh(T,X)\geq \min\left\{\tfrac{1}{2}d_\h(T,X),\tfrac{\lambda(T)}{4},\tfrac{\sep\mathcal D_{\partial T}}{2}\right\}.
\]
\end{theorem}

\begin{proof}
As we have already done in previous proofs, we can assume that $X$ is compact without loss of generality according to Lemma~\ref{lemma:H_GH_and_compactness}.
Suppose that $d_\gh(T,X)<r<\min\{\tfrac{\lambda(T)}{4},\tfrac{\sep\mathcal D_{\partial T}}{2}\}$.
Then, there is a function $f\colon T\to X\subseteq T$ with $\tfrac{1}{2}\dis(f)<r$.
The application of Proposition~\ref{prop:trees_maximal_geodesics} implies that $T\subseteq f(T)^{2r}$, and so $d_\h(T,X)\leq d_\h(T,f(T))<2r$.
\end{proof}

\subsection{Graphs with no assumptions on the subset}
\label{sub:graphs_no_assumptions}

In this subsection, we adapt the results in Section~\ref{sub:trees_no_assumptions} for graphs with loops.
The crucial tool for these adaptations is Proposition~\ref{prop:G_0_in_f(G_0)}.
We describe the results for finite graphs with leaves that are not trees, explicitly ruling out the circle case, which has been fully discussed in the pair of Theorems~\ref{thm:circleSufficient} and~\ref{thm:circleOptimal}, respectively.

\begin{proposition}
\label{prop:graphs_maximal_geodesics}
Let $G$ be a finite metric graph with leaves, which is not a tree, and let $f\colon G\to G$ be a function satisfying $\tfrac{\dis(f)}{2}<r<\min\{\tfrac{\lambda(G)}{4},\tfrac{\sep\mathcal D_{\partial G}}{2},\tfrac{e(G_0)}{24}\}$.
Then $G\subseteq f(G)^{2r}$.
\end{proposition}

\begin{proof}
The proof follows the same steps as that of Proposition~\ref{prop:trees_maximal_geodesics}.
We arrange the elements of $\mathcal D_{\partial G}$ in decreasing order $c_1>c_2>\cdots>c_n$ and define $\partial_iG$ in analogy with the previous result.
The only difference is the definition of $G_1$, which is the union of all the geodesics connecting $\partial_1G$ to $G_0$.
Then, $G_2,\dots,G_n$ are defined similarly: $G_i$ is the union of all geodesics connecting $\partial_iG$ to $\bigcup_{j=0}^{i-1}G_j$; see Figure~\ref{fig:graph_partition}.
Note that \eqref{eq:diam} still holds in this setting for $x\in G_i$ and $y\in G_j$.
We want to define bijections $\Phi_i\colon\partial_iG\to\partial_iG$ for $i\in\{1,\dots,n\}$ such that for every $v\in\partial_1G$, we have $d(f(v),\Phi_i(v))<2r$.

\begin{figure}[h!]
\centering
\begin{subfigure}{0.23\textwidth}
\centering
\begin{tikzpicture}[scale=0.72]
\draw (-2,0)--(3,0) (-1,-1.5)--(-1,1) (-1,-1)--(-1.5,-1) (0,1)--(0,-1) (1.5,-1.5)--(1.5,1.5) (1,1)--(2.5,1);
\draw (-1.5,-1)--(-1,-1) (-1,-1.5)--(-1,0)--(1.5,0)--(1.5,1)--(2.5,1) (1,1)--(1.5,1)--(1.5,1.5) (3,0)--(1.5,0)--(1.5,-1.5) (-2,0)--(-1,0)--(-1,1);
\draw (-1,-1.5) circle (3.5pt) (-1.5,-1) circle (3.5pt) (2.5,1) circle (3.5pt);
\draw (-2,0) node{\tiny$\square$};
\draw (-1,1) node{\tiny$\square$};
\draw (0,1) node{\tiny$\triangle$};
\draw (0,-1) node{\tiny$\triangle$};
\draw (1.5,1.5) node{\small$\times$};
\draw (1,1) node{\small$\times$};
\draw (3,0) node{\small$\times$};
\draw (1.5,-1.5) node{\small$\times$};
\fill[white] (0,0) circle (1.15);
\draw[red] (0,0) circle (1.15);
\fill[red,opacity=0.2] (0,0) circle (1.15);
\draw (0,0) node{$G_0$};
\end{tikzpicture}
\subcaption{}
\end{subfigure}\quad
\begin{subfigure}{0.23\textwidth}
\centering
\begin{tikzpicture}[scale=0.72]
\draw (-2,0)--(3,0) (-1,-1.5)--(-1,1) (-1,-1)--(-1.5,-1) (0,1)--(0,-1) (1.5,-1.5)--(1.5,1.5) (1,1)--(2.5,1);
\draw[ultra thick,red,dashed] (-1.5,-1)--(-1,-1) (-1,-1.5)--(-1,0)--(1.5,0)--(1.5,1)--(2.5,1);
\draw[red] (0.75,0) node[below]{$T_1$};
\draw[red] (-1,-1.5) circle (3.5pt) (-1.5,-1) circle (3.5pt) (2.5,1) circle (3.5pt);
\draw (-2,0) node{\tiny$\square$};
\draw (-1,1) node{\tiny$\square$};
\draw (0,1) node{\tiny$\triangle$};
\draw (0,-1) node{\tiny$\triangle$};
\draw (1.5,1.5) node{\small$\times$};
\draw (1,1) node{\small$\times$};
\draw (3,0) node{\small$\times$};
\draw (1.5,-1.5) node{\small$\times$};
\fill[white] (0,0) circle (1.15);
\draw[red] (0,0) circle (1.15);
\fill[red,opacity=0.2] (0,0) circle (1.15);
\draw (0,0) node{$G_0$};
\draw[red] (2.3,1) node[below]{$G_1$};
\end{tikzpicture}
\subcaption{}
\end{subfigure}\quad
\begin{subfigure}{0.23\textwidth}
\centering
\begin{tikzpicture}[scale=0.72]
\draw (-2,0)--(3,0) (-1,-1.5)--(-1,1) (-1,-1)--(-1.5,-1) (0,1)--(0,-1) (1.5,-1.5)--(1.5,1.5) (1,1)--(2.5,1);
\draw[ultra thick,red] (-1.5,-1)--(-1,-1) (-1,-1.5)--(-1,0)--(1.5,0)--(1.5,1)--(2.5,1);
\draw[ultra thick,red,dashed] (1,1)--(1.5,1)--(1.5,1.5) (3,0)--(1.5,0)--(1.5,-1.5);
\draw[red] (-1,-1.5) circle (3.5pt) (-1.5,-1) circle (3.5pt) (2.5,1) circle (3.5pt);
\draw (-2,0) node{\tiny$\square$};
\draw (-1,1) node{\tiny$\square$};
\draw (0,1) node{\tiny$\triangle$};
\draw (0,-1) node{\tiny$\triangle$};
\draw[red] (1.5,1.5) node{\small$\times$};
\draw[red] (1,1) node{\small$\times$};
\draw[red] (3,0) node{\small$\times$};
\draw[red] (1.5,-1.5) node{\small$\times$};
\draw[red] (1.5,-1) node[right]{$G_2$};
\fill[white] (0,0) circle (1.15);
\draw[red] (0,0) circle (1.15);
\fill[red,opacity=0.2] (0,0) circle (1.15);
\draw (0,0) node{$G_0$};
\end{tikzpicture}
\subcaption{}
\end{subfigure}
\quad
\begin{subfigure}{0.23\textwidth}
\centering
\begin{tikzpicture}[scale=0.72]
\draw (-2,0)--(3,0) (-1,-1.5)--(-1,1) (-1,-1)--(-1.5,-1) (0,1)--(0,-1) (1.5,-1.5)--(1.5,1.5) (1,1)--(2.5,1);
\draw[ultra thick,red] (-1.5,-1)--(-1,-1) (-1,-1.5)--(-1,0)--(1.5,0)--(1.5,1)--(2.5,1) (1,1)--(1.5,1)--(1.5,1.5) (3,0)--(1.5,0)--(1.5,-1.5);
\draw[ultra thick,red,dashed] (-2,0)--(-1,0)--(-1,1);
\draw[red] (-1,-1.5) circle (3.5pt) (-1.5,-1) circle (3.5pt) (2.5,1) circle (3.5pt);
\draw[red] (-2,0) node{\tiny$\square$};
\draw[red] (-1,1) node{\tiny$\square$};
\draw (0,1) node{\tiny$\triangle$};
\draw (0,-1) node{\tiny$\triangle$};
\draw[red] (1.5,1.5) node{\small$\times$};
\draw[red] (1,1) node{\small$\times$};
\draw[red] (3,0) node{\small$\times$};
\draw[red] (1.5,-1.5) node{\small$\times$};
\draw[red] (-1.5,0.5) node[left]{$G_3$};
\fill[white] (0,0) circle (1.15);
\draw[red] (0,0) circle (1.15);
\fill[red,opacity=0.2] (0,0) circle (1.15);
\draw (0,0) node{$G_0$};
\end{tikzpicture}
\subcaption{}
\end{subfigure}
\caption{An example of a finite graph $G$ and its partition $G_0$, $G_1$, $G_2$, $G_3$ as described in the proof of Proposition~\ref{prop:graphs_maximal_geodesics}.
The leaves in $\partial_1G$, $\partial_2G$ and $\partial_3G$ are represented with circles, crosses, and squares, respectively.
The dashed lines are $G_1$, $G_2$, and $G_3$ in (B), (C), and (D).}
\label{fig:graph_partition}
\end{figure}

According to Proposition~\ref{prop:G_0_in_f(G_0)}, $G_0\subseteq f(G_0)^r\subseteq f(G_0)^{2r}$.
Hence, if $v\in\partial G$, Lemma~\ref{lemma:ziga} gives $f(v)\notin G_0$ (with $\varepsilon=r$, since $d(v,G_0)\ge\lambda(G)>4r>r+\dis(f)$).
Both the base step and the recursive step can then be readily applied in this setting.
\end{proof}

As in the tree case, we immediately obtain lower bounds to the Gromov-Hausdorff distance.

\begin{theorem}
\label{thm:graphs_H_vs_GH_maximal_geodesics}
Let $G$ be a finite metric graph with leaves that is not a tree.
Then,
\[
d_\gh(G,X)\ge\min\big\{\tfrac{1}{2}d_\h(G,X),\tfrac{\lambda(G)}{4},\tfrac{\sep\mathcal D_{\partial G}}{2},\tfrac{e(G_0)}{24}\big\}.
\]
\end{theorem}
\begin{proof}
The proof is a straightforward adaptation of that of Theorem~\ref{thm:trees_H_vs_GH_maximal_geodesics}.
\end{proof}

\section{Conclusion and open questions}

We have furthered the study of lower bounding the Gromov--Hausdorff distance between a space and a subset by their Hausdorff distance, particularly for metric graphs.
This investigation sparks several open questions and new research directions, especially for spaces beyond graphs.

We end by listing some open questions.

\begin{question}
Is the density constant $\frac{e(G_0)}{24}$ in Theorem~\ref{thm:graphs} optimal?
\end{question}

\begin{question}
Do our results generalize to general length spaces? If so, under what density assumptions? 
\end{question}

\begin{question}
Are there versions of our results that hold for manifolds with boundary?
\end{question}

\begin{question}
Are there classes of metric spaces (other than graphs) where the Gromov--Hausdorff distance between the space and a dense enough subset equals their Hausdorff distance?
\end{question}

%
%
%

\bibliographystyle{plain}
\bibliography{LowerBoundingTheGromovHausdorffDistanceInMetricGraphs}

\appendix

\section{Another take on fractional bounds}
\label{app:fractional}

We next prove Theorems~\ref{thm:trees_H_vs_GH_with_vertex_degrees} and \ref{thm:graphs_H_vs_GH_with_vertex_degrees}, analogous results to Theorems~\ref{thm:trees_H_vs_GH_maximal_geodesics} and \ref{thm:graphs_H_vs_GH_maximal_geodesics}, that while often weaker, also provide advantages in certain cases.
We refer to Example~\ref{ex:strength_comparison} for a comparison.

Let $G$ be a graph.
We denote by $\accentset{\circ}{G}$ the graph $G$ pruned of its leaf edges.
Since $G$ is a finite graph, we note that after iterating the $\circ$ operation a finite number of times, we obtain $G_0$, the core of the graph.
For a finite graph $G$ with leaves, we denote by $\Lambda(G)$ the set of all leaf edge lengths of $G$ (so $\lambda(G)=\min\Lambda(G)$).

In the following result, we rule out the case where $T$ is a segment.
However, we have already covered the situation in Corollary~\ref{coro:segment}, which provides precise computations.

\begin{proposition}
\label{prop:trees_with_vertex_degrees}
For $T$ a finite tree, let the function $f\colon T\to T$ satisfy $\tfrac{\dis(f)}{2}<r<\min\bigl\{\tfrac{e(T)}{8},\tfrac{\sep\Lambda(T)}{5}\bigr\}$.
Assume that $T$ is not a segment, and so we can consider its canonical representation where all its non-leaf vertices $V(T)\setminus\partial T$ have degree at least $3$.
Then, $T\subseteq f(T)^{5r}$.
\end{proposition}

\begin{proof}
We  divide the proof into steps (a)--(c).

\noindent{\bf Step (a).} Since $r<\tfrac{e(T)}{6}$, for every vertex $v\in V(T)\setminus\partial T$ of degree $d$, there is a unique vertex $v'\in V(T)\setminus\partial T$ with degree at least $d$ such that $d(f(v),v')<3r$.

For every $v\in V(T)\setminus\partial T$ and every edge $e$ contributing to its degree, fix a point $x_e\in e$ at distance $d(v,x_e)=r$.
Hence, we constructed $d$ distinct points $x_e$ at distance $r$ from $v$.
Note that $d(x_e,x_{e'})=2r$ for every pair of distinct edges $e\neq e'$.
Given the bound on the distortion of $f$, $d(f(x_e),f(v))< 3r$ and $d(f(x_e),f(x_{e'}))>2r-2r=0$.

Suppose, by contradiction, that there is no vertex $v'\in V(T)\setminus\partial T$ satisfying $d(f(v),v')<3r$.
Then, thanks to the assumption on $e(T)$ and the pigeonhole principle, there are two distinct $x_e,x_{e'}$ such that the geodesic connecting $f(v)$ to $f(x_e)$ and the geodesic connecting $f(v)$ to $f(x_{e'})$ share the same initial segment of length at least $3r$.
Hence
\[
0<d(f(x_e),f(x_{e'}))\leq d(f(x_e),f(v))+d(f(v),f(x_{e'}))-2\cdot 3r<3r+3r-6r=0,
\]
a contradiction; see Figure~\ref{fig:degree_new}(A).

\begin{figure}[h!]
\centering
\begin{subfigure}{0.3\textwidth}
\centering
\begin{tikzpicture}
\draw (0,0)--(2,0)--(4,1.5) (2,0)--(4,-1.5);
\draw[dashed] (-0.5,0)--(0,0);
\fill (3.5,1.125) circle (2pt) (3.5,-1.125) circle (2pt) (0.5,0) circle (2pt);
\draw (0.5,-0.2)--(0.5,-0.4)--(2,-0.4)node[pos=0.5,below]{$\geq 3r$}--(2,-0.2);
\draw (0.5,0) node[above]{$f(v)$};
\draw (3.5,1.125) node[above left]{$f(x_e)$};
\draw (3.5,-1.125) node[above right]{$f(x_{e'})$};
\end{tikzpicture}
\caption{}
\label{fig:degreeA}
\end{subfigure}
\begin{subfigure}{0.68\textwidth}
\centering
\begin{tikzpicture}[scale=0.75]
\draw (0,0)--(2,2) (0,0)--(-2,2) (0,0)--(2,-2) (0,0)--(-2,-2);
\fill (0,0) circle (2.67pt) (1.5,1.5) circle (2.67pt) (-1.5,1.5) circle (2.67pt) (1.5,-1.5) circle (2.67pt) (-1.5,-1.5) circle (2.67pt);
\draw (0,0) node[above]{$v$};
\draw (1.5,1.5) node[above left]{$x_{e_1}$};
\draw (-1.5,1.5) node[above right]{$x_{e_2}$};
\draw (1.5,-1.5) node[below left]{$x_{e_3}$};
\draw (-1.5,-1.5) node[below right]{$x_{e_4}$};
\draw (5.5,0)--(3.5,2) (6.5,0)--(5.5,0)--(3.5,-2) (6.5,0)--(8.5,-2) (6.5,0)--(8.5,2);
\fill (6,0) circle (2.67pt) (4,1.5) circle (2.67pt) (8,1.5) circle (2.67pt) (4,-1.5) circle (2.67pt) (8,-1.5) circle (2.67pt);
\draw (6,0) node[above]{$f(v)$};
\draw (8,1.5) node[above left]{$f(x_{e_1})$};
\draw (4,1.5) node[above right]{$f(x_{e_2})$};
\draw (8,-1.5) node[below left]{$f(x_{e_3})$};
\draw (4,-1.5) node[below right]{$f(x_{e_4})$};
\draw (5.5,-0.2)--(5.5,-0.4)--(6,-0.4)node[below]{$<3r$}--(6,-0.2);
\end{tikzpicture}
\caption{}
\label{fig:degreeB}
\end{subfigure}
\caption{(A) If the geodesics going from $f(v)$ to $f(x_e)$ and $f(x_{e'})$ agree for at least $3r$ of their length, then the distance between $f(x_e)$ and $f(x_{e'})$ is at most $d(f(v),f(x_e))+d(f(v),f(x_{e'}))-2\cdot 3r$.
(B) The assumption on $e(G)$ is essential.
If there were shorter edges, a vertex $v$ of degree $4$ could be sent to a point $3r$-close to two vertices of degree $3$ without contradiction.}
\label{fig:degree_new}
\end{figure}

As for uniqueness, if $v',v''\in V(T)\setminus\partial T$ satisfy $d(f(v),v')<3r$ and $d(f(v),v'')<3r$, then $d(v',v'')<6r$, which implies that $v'=v''$ by the constraint on $e(T)$.

\smallskip

\noindent{\bf Step (b).} Assume that $r<\tfrac{e(T)}{8}$.
The function $\Psi\colon V(T)\setminus\partial T\to V(T)\setminus\partial T$ associating $v\mapsto v'$ in the notation of {\bf Step (a)} is a bijection that preserves the vertex degrees.

Let $v,w\in V(T)\setminus\partial T$ be such that $\Psi(v)=\Psi(w)$.
Then
\[
d(v,w)\leq d(f(v),f(w))+\dis(f)\leq d(f(v),\Psi(v))+d(\Psi(w),f(w))+\dis(f)<3r+3r+2r=8r,
\]
and so $v=w$ given the constraint on $e(T)$.
The fact that $\Psi$ preserves the vertex degrees follows from the finiteness of $T$ and the pigeonhole principle; see Figure~\ref{fig:degree_new}(B).

\smallskip

\noindent{\bf Step (c).} Assume that $r<\min\{\tfrac{e(T)}{8},\tfrac{\sep\Lambda(T)}{5}\}$.
Then, there is a bijection $\Theta\colon\partial T\to\partial T$ such that for every $v\in\partial T$, we have $d(f(v),\Theta(v))<5r$.

The proof idea is similar to that of Proposition~\ref{prop:trees_maximal_geodesics}.
Let us arrange the elements of $\Lambda(T)$ in decreasing order: $\lambda_1>\lambda_2>\cdots>\lambda_n$.
Partition the set $\partial T$ accordingly, in subsets $\partial_iT$, for $i\in\{1,\dots,n\}$, collecting the leaves whose leaf edges have length $\lambda_i$.
We define $n$ bijections $\Theta_i\colon\partial_iT\to\partial_iT$ that form $\Theta$ by induction.

Let us fix some notation.
Let us partition the leaf edges into $E_1,\dots,E_n$ according to the partition $\partial_1T,\dots,\partial_nT$ of the leaves.
For every $e\in E_i$, we denote by $v_e$ its leaf and by $w_e$ its end in $V(T)\setminus\partial T$.

\smallskip

\noindent{\bf Base step.} For every $v_e\in\partial_1T$, we have $f(v_e)\notin \accentset{\circ}{T}$.
Indeed, according to {\bf Step (b)}, $\accentset{\circ}T\subseteq f(\accentset{\circ}T)^{3r}$ since the latter is path connected and contains all the vertices in $T\setminus\partial T$.
Thus, Lemma~\ref{lemma:ziga} implies the claim by setting $\varepsilon=3r$ since $d(v_e,\accentset{\circ}T)\geq e(T)>8r$.

The image $f(e)^{3r}$ is path connected and contains both $f(v_e)$ and $\Psi(w_e)$.
We claim that $\Psi(w_e)$ is the only vertex in $V(T)\setminus\partial T$ that is covered by $f(e)^{3r}$.
Indeed, suppose that $w\in V(T)\setminus\partial T$ satisfies $\Psi(w)\in f(e)^{3r}$, and let $x\in e$ be such that $d(f(x),\Psi(w))<3r$.
Then,
\[
\begin{aligned}
d(x,w)&\,\leq d(f(x),f(w))+\dis(f)\leq d(f(x),\Psi(w))+d(\Psi(w),f(w))+\dis(f)\\
&\,<3r+3r+2r=8r<e(T),
\end{aligned}
\]
which implies that $w=w_e$ since any geodesic connecting $x$ to a point in $V(T)\setminus\partial T$ must cross $w_e$.

Furthermore, note that
\[
\begin{aligned}
d(f(v_e),\Psi(w_e)&\,\geq d(f(v_e),f(w_e))-d(f(w_e),\Psi(w_e))\geq d(v_e,w_e)-\dis(f)-d(f(w_e),\Psi(w_e))\\
&\,>\lambda_1-2r-3r=\lambda_1-5r>\lambda_1-\sep\Lambda(T)\geq\lambda_2.
\end{aligned}
\]
It follows that $f(v_e)$ must belong to an edge $e'\in E_1$ since the other leaf edges are too short.
Thus, $w_{e'}=\Psi(w_e)$ and
\[
d(v_{e'},f(v_e))=d(v_{e'},w_{e'})-d(w_{e'},f(v_e))=\lambda_1-d(\Psi(w_e),f(v_e))<\lambda_1-(\lambda_1-5r)=5r.
\]
Set $\Theta_1(v_e)=v_{e'}$.

We claim that the function $\Theta_1$ is injective.
Let $v_e,v_{e'}\in \partial_1T$ such that $\Theta_1(v_e)=\Theta_1(v_{e'})$.
Then,
\[
d(v_e,v_{e'})\leq d(f(v_e),\Theta_1(v_e))+d(\Theta_1(v_{e'}),f(v_{e'}))+\dis(f)<5r+5r+2r=12r<2e(T),
\]
and so $v_e=v_{e'}$.
Since $\partial_1T$ is finite, $\Theta_1$ is a bijection.

\smallskip

\noindent{\bf Recursive Step.} Suppose that $\Theta_1,\dots,\Theta_i$ have been constructed with the desired properties.
Let $e\in E_{i+1}$.
We claim that $f(v_e)\notin\accentset{\circ}T\cup E_1\cup\cdots\cup E_i$.
As in the base step, $f(v_e)\notin\accentset{\circ}T$.
Furthermore, $d(v_e,E_j)\geq e(T)>8r$ and $E_j\subseteq f(E_j)^{5r}$ for every $j\in\{1,\dots,i\}$.
Then, the claim follows thanks to Lemma~\ref{lemma:ziga}.

We can prove as in the base step that there is a leaf edge $e'\in E_{i+1}$ such that $f(v_e)\in e'$ due to the constraint on $\sep\Lambda(T)$, and the function $\Theta_{i+1}\colon v_e\mapsto v_{e'}$ is well-defined, bijective and satisfies $d(f(v_e),\Theta_{i+1}(e))<5r$.

\smallskip

As a consequence of {\bf Step (c)}, $f(T)^{5r}$ contains all the vertices of $T$.
Therefore, it contains $T$ entirely since $f(T)^{5r}$ is path connected.
\end{proof}

As a corollary of Proposition~\ref{prop:trees_with_vertex_degrees}, we obtain the following result.

\begin{theorem}
\label{thm:trees_H_vs_GH_with_vertex_degrees}
Let $T$ be a finite tree that is not a segment and let $X$ be a subset thereof.
Consider the canonical representation of $T$ where every non-leaf vertex has degree at least $3$.
Then
\[
d_\gh(T,X)\geq\min\left\{\tfrac{1}{5}d_\h(T,X),\tfrac{e(T)}{8},\tfrac{\sep\Lambda(T)}{5}\right\}.
\]
\end{theorem}

\begin{proof}
The proof is the same as that of Theorem~\ref{thm:trees_H_vs_GH_maximal_geodesics}, except now using Proposition~\ref{prop:trees_with_vertex_degrees} in place of Proposition~\ref{prop:trees_maximal_geodesics}.
\end{proof}

As mentioned above, the two results, Theorem~\ref{thm:trees_H_vs_GH_maximal_geodesics} and~\ref{thm:trees_H_vs_GH_with_vertex_degrees}, do not imply one another and can be used in different situations to provide better lower bounds.

\begin{example}
\label{ex:strength_comparison}
Consider the tree $T$ and subset $X$ described in Theorem~\ref{thm:counterExample}.
By applying Theorem~\ref{thm:trees_H_vs_GH_maximal_geodesics}, we obtain that $d_\gh(T,X)\geq\frac{1}{2}d_\h(T,X)=\frac{\varepsilon}{2}$, whereas Theorem~\ref{thm:trees_H_vs_GH_with_vertex_degrees} only provides the weaker lower bound $d_\gh(T,X)\geq\min\{\frac{\varepsilon}{5},\frac{1}{8}\}$.

For $\varepsilon<\frac{1}{4}$, let us modify the same example by attaching to each leaf of the tree $T$ a pair of new leaf edges of length $1$.
Let us denote by $T'$ the tree constructed in this way.
Let $X'$ be a subset of $T'$ satisfying $d_\h(T',X')>\frac{5\varepsilon}{2}$.
Note that $\lambda(T')=1$ and $\sep\mathcal D_{\partial T'}=\varepsilon$.
Therefore, Theorem~\ref{thm:trees_H_vs_GH_maximal_geodesics} implies that
\[
d_\gh(T',X')\geq \min\big\{\tfrac{d_\h(T',X')}{2},\tfrac{1}{4},\tfrac{\varepsilon}{2}\big\}=\tfrac{\varepsilon}{2}.
\]
However, since $e(T')=1$ and $\Lambda(T')=\{1\}$, which implies $\sep\Lambda(T')=\infty$, using Theorem~\ref{thm:trees_H_vs_GH_with_vertex_degrees} we obtain the stronger inequality
\[
d_\gh(T',X')\geq\min\big\{\tfrac{1}{5}d_\h(T',X'),\tfrac{1}{8}\big\}>\tfrac{\varepsilon}{2}.
\]
\end{example}

We again have analogous results for graphs instead of trees.

\begin{proposition}
\label{prop:graphs_with_vertex_degrees}
Let $G$ be a finite graph with leaves that is not a tree, and consider its canonical representation where each non-leaf vertex has degree at least $3$.
Let  
$f\colon G\to G$ satisfy $\tfrac{\dis(f)}{2}<r<\min\{\tfrac{e(G)}{8},\tfrac{\sep\Lambda(G)}{5},\tfrac{e(G_0)}{24}\}$.
Then, $G\subseteq f(G)^{5r}$.
\end{proposition}

\begin{proof}
The proof is built upon Propositions~\ref{prop:G_0_in_f(G_0)} and~\ref{prop:trees_with_vertex_degrees}.
Both {\bf Step (a)} and {\bf Step (b)} can be followed similarly.
In the notation of the proof of {\bf Step (a)} in Proposition~\ref{prop:trees_with_vertex_degrees}, if $e$ is a self-loop contributing twice to the vertex degree, we pick a point for each end of $e$.
For the sake of simplicity, these pairs of points can be treated as coming from distinct edges.
Note that $d(x_e,x_{e'})=2r$ for every pair of distinct edges $e\neq e'$ since $e(G)>4r$.
As for {\bf Step (c)}, we still partition the leaves $\partial G=\partial_1G\cup\cdots\cup\partial_nG$ and the leaf edges $E_1\cup\cdots\cup E_n$ according to the length of the leaf edges.
Then, the bijective functions $\Theta_1,\dots,\Theta_n$ can be constructed in the same way with the further observation that a leaf $v\in\partial G$ cannot be sent to $G_0$ since the latter is contained in $f(G_0)^r$.
Denote by $F$ the forest consisting of the union of geodesics connecting $V(G)\setminus(\partial G\cup V(G_0))$ to $G_0$.
Thus,
\[
G=G_0\cup F\cup E_1\cup\cdots\cup E_n\subseteq f(G_0)^r\cup f(\accentset{\circ}{G})^{3r}\cup f(E_1)^{5r}\cup\cdots\cup f(E_n)^{5r}\subseteq f(G)^{5r}.
\]
\end{proof}

\begin{theorem}
\label{thm:graphs_H_vs_GH_with_vertex_degrees}
Let $G$ be a finite metric graph with leaves that is not a tree.
Assume that $G$ is in its canonical representation where each non-leaf vertex has degree at least $3$.
Then,
\[
d_\gh(G,X)\geq \min\big\{\tfrac{1}{5}d_\h(G,X),\tfrac{e(G)}{8},\tfrac{\sep\Lambda(G)}{5},\tfrac{e(G_0)}{24}\big\}.
\]
\end{theorem}
\begin{proof}
The proof is a straightforward adaptation of that of Theorem~\ref{thm:trees_H_vs_GH_with_vertex_degrees}.
\end{proof}

As in the tree case (Example~\ref{ex:strength_comparison}), it is possible to show that Theorems~\ref{thm:graphs_H_vs_GH_maximal_geodesics} and~\ref{thm:graphs_H_vs_GH_with_vertex_degrees} provide independent lower bounds.\\

\end{document}